\documentclass[11pt, oneside, reqno]{amsart}
\newcommand{\hkra}{\hookrightarrow}

\newcommand{\set}[1]{\left\{ #1 \right\}}

\newcommand{\cc}[1]{\overline{#1}}

\newcommand{\sub}{\subset}

\newcommand{\sm}{\ensuremath{\setminus}}

\newcommand{\s}[2]{\sum\limits_{#1}{#2} }
\newcommand{\p}[2]{\prod\limits_{#1}{#2} }
\newcommand{\g}{\circ}
\newcommand{\modulo}[2]{{\raisebox{.2em}{$#1$}\left/\raisebox{-.2em}{$#2$}\right.}}
\newcommand{\sign}{\normalfont\text{sign}}
\newcommand{\inv}{^{-1}}

\newcommand{\cl}{\colon}
\newcommand{\djun}[1]{\bigsqcup\limits_{#1}}

\newcommand{\lbr}[1]{\Bigl(#1\Bigr)}
\newcommand{\emst}{\emptyset}
\newcommand{\iI}{_{i \in I}}

\newcommand{\xra}{\xrightarrow}

\newcommand{\scA}{\mathscr{A}}

\newcommand{\scF}{\mathscr{F}}

\newcommand{\conn}{\nabla}
\newcommand{\eva}{ev}

\newcommand{\lspan}[1]{\langle {#1}\rangle}

\newcommand{\lc}{\underline}

\newcommand{\ide}{\normalfont\text{id}}

\newcommand{\PGL}{\normalfont\text{PGL}}

\newcommand{\coker}{\normalfont\text{coker}}
\newcommand{\pr}{\normalfont\text{pr}}

\newcommand{\Mbar}{\overline{\mathcal M}}

\newcommand{\delbar}{\bar\partial}
\newcommand{\del}{\partial}
\newcommand{\Hom}{\normalfont\text{Hom}}

\newcommand{\Aut}{\normalfont\text{Aut}}

\newcommand{\bC}{\mathbb{C}}
\newcommand{\bD}{\mathbb{D}}

\newcommand{\bL}{\mathbb{L}}

\newcommand{\bN}{\mathbb{N}}

\newcommand{\bP}{\mathbb{P}}

\newcommand{\bR}{\mathbb{R}}

\newcommand{\bZ}{\mathbb{Z}}
\newcommand{\cA}{\mathcal{A}}
\newcommand{\cB}{\mathcal{B}}
\newcommand{\cC}{\mathcal{C}}

\newcommand{\cE}{\mathcal{E}}

\newcommand{\cG}{\mathcal{G}}

\newcommand{\cJ}{\mathcal{J}}
\newcommand{\cK}{\mathcal{K}}

\newcommand{\cM}{\mathcal{M}}

\newcommand{\cO}{\mathcal{O}}

\newcommand{\cT}{\mathcal{T}}
\newcommand{\cU}{\mathcal{U}}

\newcommand{\cZ}{\mathcal{Z}}

\newcommand{\fB}{\mathfrak{B}}
\newcommand{\fC}{\mathfrak{C}}

\newcommand{\fb}{\mathfrak{b}}

\newcommand{\fii}{\mathfrak{i}}

\newcommand{\fm}{\mathfrak{m}}

\newcommand{\fp}{\mathfrak{p}}
\newcommand{\fq}{\mathfrak{q}}

\newcommand{\fs}{\mathfrak{s}}

\newcommand{\pcd}{\normalfont\text{PD}}
\newcommand{\wt}{\widetilde}

\newcommand{\wh}{\widehat}

\newcommand{\evai}{evi}
\newcommand{\evab}{evb}

\newcommand{\orl}{\normalfont\text{or}}

\newcommand{\al}{\alpha}
\newcommand{\ga}{\gamma}

\newcommand{\OC}{\mathcal{OC}}

\newcommand{\qb}{\mathfrak{q}^{b,\ga}}
\newcommand{\mb}{\mathfrak{m}^{b,\ga}}
\newcommand{\m}{\mathfrak{m}}
\newcommand{\q}{\mathfrak{q}}
\newcommand{\CC}{\mathbb{C}}

\newcommand{\RR}{\mathbb{R}}
\newcommand{\ZZ}{\mathbb{Z}}

\newcommand{\LL}{\mathbb{L}}
\newcommand{\mfa}{\mathsf{a}}
\newcommand{\obs}{s}

\newcommand{\ogw}{\normalfont\text{OGW}}
\newcommand{\bigplus}{%
	\DOTSB\mathop{\mathpalette\mattos@bigplus\relax}\slimits@
}
\newtheorem{intthm}{Theorem}

\newcommand{\dmc}{\overline{\mathscr{O}\mathscr{C}}}

\newcommand{\stb}{\normalfont\text{st}}

\newcommand{\ov}[1]{\overline{#1}}
\newcommand{\cKc}{\cK^{\scale{\square}{0.6}}}
\newcommand{\scale}[2]{\scaleobj{#2}{#1}}
\newcommand{\cBc}{\cB^{\scale{\square}{0.6}}}
\newcommand{\cEc}{\cE^{\scale{\square}{0.6}}}
\newcommand{\cTc}{\cT^{\scale{\square}{0.6}}}
\newcommand{\bdr}{o}

\newcommand{\cMc}{\Mbar^{\scale{\square}{0.6}}}

\usepackage[utf8]{inputenc}
\usepackage[english]{babel}
\usepackage{amsmath} 
\usepackage{amssymb}
\usepackage{amsfonts}
\usepackage{mathtools}
\usepackage{amsthm}
\usepackage{scalerel}[2016/12/29]
\usepackage[shortlabels]{enumitem}
\usepackage[dvipsnames]{xcolor}
\usepackage{tikz-cd}
\usepackage{subfiles}
\usetikzlibrary{tqft}
\usepackage[colorinlistoftodos]{todonotes}
\usepackage[mathscr]{euscript}
\usepackage{extarrows}
\usepackage[pagebackref=true]{hyperref}
\usepackage{comment}
\usepackage{appendix}
\usepackage{bm}
\hypersetup{
	colorlinks=true,
	linkcolor=Blue,
	filecolor=red,      
	urlcolor=teal,
	citecolor=NavyBlue,}

\DeclareFontFamily{U}{rcjhbltx}{}
\DeclareFontShape{U}{rcjhbltx}{m}{n}{<->rcjhbltx}{}
\DeclareSymbolFont{hebrewletters}{U}{rcjhbltx}{m}{n}
\DeclareMathSymbol{\mem}{\mathord}{hebrewletters}{109}
\DeclareMathSymbol{\nun}{\mathord}{hebrewletters}{110}

\newtheorem{theorem}{Theorem}[section]
\newtheorem{lemma}[theorem]{Lemma}
\newtheorem{corollary}[theorem]{Corollary}
\newtheorem{cor}[theorem]{Corollary}
\newtheorem{proposition}[theorem]{Proposition}

\theoremstyle{definition}
\newtheorem{definition}[theorem]{Definition}
\theoremstyle{remark}
\newtheorem{remark}[theorem]{Remark}

\newtheorem{example}[theorem]{Example}

\newtheorem*{notation*}{Notation}

\numberwithin{equation}{section}

\makeatletter
\@addtoreset{equation}{section}
\@addtoreset{theorem}{section}
\makeatother

\newcommand{\Addresses}{{
		\bigskip
		\footnotesize
		\textsc{A. Hirschi, Université Paris Cité, Sorbonne Université, CNRS, IMJ-PRG, F-75005 Paris, France}\par\nopagebreak
		\text{ORCID}: \texttt{0000-0002-2392-7875}\\
		
		\textsc{K. Hugtenburg, School of Mathematics and Statistics, University of St Andrews}\par\nopagebreak
		\text{ORCID}: \texttt{0000-0002-7823-7126}
}}
\makeatletter 
\def\l@subsection{\@tocline{2}{0pt}{2pc}{6pc}{}} \makeatother

\usepackage{microtype}  
\begin{document}
	
	\title[Open Gromov--Witten invariants and Lagrangian cobordisms]{Open Gromov--Witten invariants in genus zero and Lagrangian cobordisms}
	\author{Amanda Hirschi}\author{Kai Hugtenburg}
	\date{today}
	\begin{abstract} We construct open Gromov--Witten invariants in genus zero for arbitrary closed symplectic manifolds and embedded relatively spin Lagrangians which are weakly unobstructed by a bounding cochain. This uses the foundational work of \cite{HH25,HH26} and the algebraic framework of \cite{ST21}. We prove the open WDVV relations and show that these invariants are independent of the choice of almost complex structure and under Hamiltonian isotopy. We also prove a relation between open Gromov--Witten invariants of cobordant Lagrangians.
	\end{abstract}
	
	\maketitle
	\tableofcontents

	\section{Introduction}
\subsection{Context}
Let $(X,\omega)$ be a smooth closed symplectic manifold, with $L \subset X$ a closed embedded relatively spin, Lagrangian submanifold. The history of associating algebraic structures to holomorphic discs $(D^2, S^1) \rightarrow (X,L)$ goes back to \cite{Fuk93}. In \cite{FOOO18}, such holomorphic discs are used to construct the structure of an $A_\infty$ algebra on (singular) chains on $L$. The $A_\infty$ structure comes from the fact that the moduli spaces $\Mbar_{k,\ell}^{J,\beta}(X,L)$ of holomorphic discs have codimension $1$ boundary strata formed by disc bubbling. When one wants to define \emph{open Gromov-Witten invariants}, which would count holomorphic discs, it is this disc bubbling, together with the codimension $1$ boundary, where the boundary of the discs is collapsed to a point, thus forming a holomorphic sphere through $L$ (which only occurs when $k=0$), which ruins any kind of invariance of the count.
\par
Welschinger was the first to solve these issues and define such invariants in low dimensions, in the revolutionary work \cite{Wel05} and afterwards in \cite{Wel13,Wel15}. Liu \cite{Liu20} defines open Gromov-Witten invariants in arbitrary genus in the presence of an $S^1$ action on $X$, which leaves $L$ invariant. Moreover, \cite{Fuk11} and \cite{FOOO10,FOOO11} invent the notion of a bounding cochain, which is a solution for the Maurer-Cartan equation for the $A_\infty$ algebra associated to $L$, and use these to algebraically deal with disc bubbling. As shown in \cite{Fuk11} \cite{Joy08}, when $[L] = 0 \in H_n(X)$, one can additionally include counts of holomorphic spheres passing through a chain $C$ with $\partial C = L$ to obtain open Gromov-Witten invariants. 

In \cite{ST16}, \cite{ST21}, the theory of bounding cochains is used to define genus $0$ open Gromov-Witten invariants for weakly unobstructed Lagrangians under restrictive geometric assumptions. In particular, they assume that $\Mbar_{k,\ell}^{J,\beta}(X,L)$, the moduli space of $J$-holomorphic discs with boundary on $L$, is a smooth orbifold with corners. For a general symplectic manifold, this condition does not hold.  In this paper we use the presentation of $\Mbar_{k,\ell}^{J,\beta}(X,L)$ as a global Kuranishi chart, obtained in \cite{HH25}, together with the constructions of \cite{HH26} and \cite{ST21} to arrive at a general definition of genus $0$ open Gromov-Witten invariants.

\subsection{Main results}
Let $\Lambda$ be the Novikov ring \[
	\Lambda \coloneqq  \set{  \sum_{i = 0}^{\infty}a_iQ^{\beta_i}\,\big|\, a_i \in \mathbb{R},\; \beta_i \in H_2(X, L), \; \lim_{i\to \infty } \omega(\beta_i) = \infty }.
\]
Analogously to \cite{ST16}, we define $\fq$-operations associated to discs with $k$ incoming boundary marked points, $\ell$ incoming interior marked points, and a single outgoing boundary marked point.\begin{equation}\label{eq:q-operations-intro}
	\fq_{k,\ell}\cl  \Omega^*(L;\Lambda)^{\otimes k} \otimes \Omega^*(X;\Lambda)^{\otimes \ell} \rightarrow \Omega^*(L;\Lambda) 
	\end{equation}
	We show in \ref{sec:props-of-q-ops} that they satisfy the same properties and structure relations as in \cite{ST16}. In particular, we obtain

\begin{proposition}
		\label{int:a_infty-algebra}
	The operations $\{\fq_{k,0}\}_{k\geq 0}$ equip $\Omega^*(L;\Lambda)$ with a cyclic unital filtered $A_\infty$ structure. It is independent up to pseudoisotopy of the choice of
		\begin{itemize}[leftmargin=20pt]
			\item almost complex structure $J$
			\item auxiliary data used to define the global Kuranishi charts
			\item Thom system
			\item $L$ in its Hamiltonian isotopy class.
		\end{itemize}
	\end{proposition}

A similar result first appeared as \cite[Theorem~12.1]{Fu10}. In our setting, we obtain the full $A_\infty$ algebra directly, without having to truncate the energy of holomorphic discs. Moreover, our operations are compatible with interior forgetful maps.

The $\fq$-operations of~\eqref{eq:q-operations-intro} are part of the more general structure of an open-closed Deligne--Mumford Field Theory (DMFT) defined in \cite{HH26} based on the geometric foundations of \cite{HH25}. In this paper we are only concerned with the genus $0$ (operadic) part of this theory where curves have at most one boundary component and at most one output marked point. In particular, the open-closed DMFT of \cite{HH26} contains all the required operations to carry out the construction of open Gromov--Witten invariants in genus as given by \cite{ST21} and \cite{ST23} in the regular case. 

We recall the required additional inputs for their construction; for more detail refer to \textsection\ref{sec:OGW construction}. Throughout, $(X,\omega)$ is closed and $L \subset X$ is a nonempty embedded Lagrangian submanifold, which is weakly unobstructed by a point-like bounding cochain over some subspace $W \subset \hat{H}^*(X,L)$, cf Definition~\ref{def:weak-bounding-pair}. Let 
$$\hat{H}^*(X,L) = H^*\left(\ker\left( \int_L\cl \Omega^*(X) \rightarrow \RR \right) \right)$$ 
be a modification of the relative cohomology of $L$ in $X$. If $[L] = 0 \in H_n(X;\RR)$ choose a map 
\begin{equation}
	P\cl \hat{H}^*(X,L) \rightarrow \RR
\end{equation}
splitting the map $y\cl \RR \rightarrow \hat{H}^*(X,L)$ in the long exact sequence induced by the surjective map $\int_L \cl \Omega^*(X)\to\bR$. 

\begin{intthm}
	\label{intthm:OGW}
	Suppose $L$ is equipped with a weak bounding pair $(b,\gamma)$ (see Definition \ref{def:weak-bounding-pair}).
	\begin{itemize}[leftmargin=20pt]
		\item If $[L] \neq 0$, there exists invariants
		\begin{equation}\label{eq:ogw-intro}
			\normalfont\text{OGW}_{\beta,0,\ell}\cl W^{\otimes \ell} \rightarrow \RR.
		\end{equation}
		 whenever $\beta$ satisfies $\partial \beta \neq 0 \in H_1(L;\ZZ)$. Moreover, if $(b,\gamma)$ is point-like, there also exist invariants $\normalfont\text{OGW}_{\beta,k,\ell}$ for $k \ge 1$ and any $\beta\in H_2(X,L;\bZ)$.
		\item If $[L] = 0$, any choice of $P$ as above determines for all $\beta \in H_2(X,L;\ZZ)$ genus-zero open Gromov--Witten invariants $\normalfont\text{OGW}_{\beta,0,\ell}^P$ as in~\eqref{eq:ogw-intro}. Similarly, if $b$ is point-like, there exist invariants $	\normalfont\text{OGW}_{\beta,k,\ell}$ for $k \ge 1$ not depending on $P$.	
	\end{itemize} 
	In either case, the invariants are independent of the choice of
	\begin{itemize}
		\item almost complex structure $J\in \cJ_\tau(X,\omega)$,
		\item representative of the Hamiltonian isotopy class of $L$,
		\item the gauge equivalence class of the bounding pair.
	\end{itemize}
Moreover, they agree with the invariants of \cite{ST21} whenever the latter are defined.
\end{intthm}

Intuitively, ignoring the bounding cochain, $\normalfont\text{OGW}^P_{\beta, k, \ell}(\eta)$ for $\eta \in \hat{H}^*(X,L)^{\otimes k}$ is the count of the number of discs with boundary on $L$ in homology class $\beta$ with $k$ boundary marked points constrained on $k$ generic points, and $\ell$ interior marked points constrained on cycles $\pcd(\eta_i) \in H_*(X \setminus L)$ together with the count of the number of spheres with $\ell + 1$ marked points, $\ell$ of which are constrained to the cycles $PD(\eta_i)$, and the remaining one by the cycle $C \in H_{n+1}(X,L)$ with $\partial C = [L]$ determined by the map $P$.

As this intuitive description suggests, even though our invariants are defined over the real numbers, one would expect them to be rational numbers (or algebraic numbers, when one is working with non-trivial local systems on $L$). Progress on this has been achieved in \cite{Han24}.

\begin{remark}\label{rem:choice of P}
	The choice of $P$ is equivalent to the choice of $C  \in H_{n+1}(X,L;\RR)$ such that $\partial C = [L]$ (\cite[Remark~4.12]{ST23}). The choice of $P$ only affects the invariants $\text{OGW}^P_{\beta,k,\ell}$ with $k = 0$ and $\partial \beta = 0 \in H_1(L;\ZZ)$. Moreover, any two choices $P$, $P'$ are related by a class $Q \in H_n(X;\RR)$, and $$\text{OGW}^{P'}_{\beta,0,\ell} - \text{OGW}^{P}_{\beta,0,\ell} = \text{GW}_{\beta,\ell+1}(PD(Q), \_)$$ 
	by \cite[Proposition~8.1]{Hollands_2025}.
\end{remark}

\begin{remark}
	Invariance under gauge equivalence of the bounding pair was proved in \cite{ST21} whenever their invariants are defined. Their proof caries over verbatim to our setting. \cite{ST21} only showed invariance under deformation of almost complex structures in the case where the moduli spaces $\Mbar_{k,\ell}^{J,\beta}(X,L)$ remain smooth.
\end{remark}

In general, computing open Gromov--Witten invariants is a non-trivial task. Solomon--Tukachinsky \cite[Theorem~3]{ST23} establish recursive relations for their invariants, called the open WDVV equations in analogy with the WDVV equation for usual closed Gromov--Witten invariants. These relations have been used for the computation of all open Gromov--Witten invariants for certain real loci in complete intersections in projective space, and the Chiang Lagrangian in $\mathbb{C}P^3$, \cite{ST23, HT2024examplesrelativequantumcohomology, Hollands_2025,BT26}. We establish the open WDVV equations in general, which will allow these computational techniques to be applied in a wider range of settings.

\begin{intthm}[Theorem~\ref{thm:owdvv-equation}]
	The open Gromov--Witten invariants of Theorem \ref{intthm:OGW} satisfy the open Gromov--Witten axioms of~\ref{prop:OGW-axioms} and the open WDVV equation, cf. \eqref{eq:owdvv}.
\end{intthm}

\subsubsection{Lagrangian branes} 
A \emph{generalised Lagrangian brane} $\LL$ is a union of (not necessarily transversely intersecting) Lagrangians. The construction of the open Gromov--Witten invariants of Theorem~\ref{intthm:OGW} can be extended to generalised Lagrangian branes, described in \S\ref{sec:OGW construction}. In this context, we can prove a stronger invariance statement than in Theorem~\ref{intthm:OGW}. It relies on the notion of Lagrangian cobordisms studied systematically by \cite{BC13,NT20}. Any Hamiltonian isotopy induces a Lagrangian cobordism, but there are also more interesting examples, \cite{Hau20,Hic23}, such as those coming from mutations.

Suppose $\wt L \subset X \times \CC$ is a Lagrangian cobordism as in~Definition \ref{def:Lag cob} between two generalised Lagrangian branes $\LL^+$ and $\LL^-$. Note that while we work with non-empty Lagrangians in $X$, a Lagrangian brane could be the empty sequence. Assume there exists a bounding cochain $\wt b$ for $\wt L$, such that $\wt b|_{\LL^{\pm}}$ is point-like. Let $\normalfont\text{Cone}(\LL) := \text{Cone}(\int_\LL\cl \Omega^*(X) \rightarrow \RR)$, and note that a map $P\cl \hat{H}^*(X,\LL) \rightarrow \RR$ is equivalent to a map $P\cl H^*(\normalfont\text{Cone}(\LL)) \rightarrow \RR$. Then define 
\begin{align*}
	\Phi\cl \normalfont\text{Cone}(\LL^-) &\;\rightarrow\; \normalfont\text{Cone}(\LL^+)\\ \notag
	(\alpha, a) &\;\mapsto\; (\alpha, a + \int_{\wt L} \pi_X^* \alpha),
\end{align*}
where $\pi_X\cl X \times \CC \rightarrow X$ is the projection.

\begin{intthm}
	\label{intthm:cobordism}
	Assume the invariants ${\normalfont\text{OGW}}^{\LL^+}$ are defined using the section $P^+$. Then, for each $\beta \in H_2(\wt X, \wt L)$ and $k \geq 0$, we have 
	\begin{equation}
		\sum_{f^+_*\beta_+ =\beta} {\normalfont\text{OGW}}^{\LL^+}_{\beta_+,k,\ell} = \sum_{f^-_*\beta_- =\beta}{\normalfont\text{OGW}}^{\LL^-}_{\beta_-,k,\ell}
	\end{equation}
	as maps $W^{\otimes \ell} \rightarrow \RR,$	where the invariants ${\normalfont\text{OGW}}^{\LL^-}$ are defined using $P^- := P^+ \circ \Phi$, and $f^{\pm}$ is the inclusion $(X, \LL_\pm) \hookrightarrow (X\times \CC, \wt L)$.
\end{intthm}

In \cite{ST23}, these invariants are denoted $\overline{\normalfont\text{OGW}}$ to indicate that these definition requires the choice of a map $P_\bR$.

\begin{remark}\label{rem:bc-on-cobordism}
	The existence of a weak bounding cochain on $\wt L$ extending point-like bounding cochains on its ends is a strong constraint. It could also happen that a bounding cochain $\wt b$ only exists over a smaller $\wt W \subset W$, or that $\wt b|_{\LL^{\pm}}$ is a non-point-like bounding cochain. In these cases, only invariants with interior markings from $\wt W$, respectively invariants without boundary marked points, are preserved under the cobordism.
\end{remark}

Let $\normalfont\text{Cob}^{wbc}_0(X)$ denote the group of null-homologous generalised Lagrangian branes up to weakly unobstructed Lagrangian cobordism. We can use the above result to detect non-trivial Lagrangian cobordism classes, restricting to invariants without boundary marked points by the issues of Remark~\ref{rem:bc-on-cobordism}

\begin{cor}
	If $\bL$ is a Lagrangian brane with $[\LL] = 0 \in \normalfont\text{Cob}^{wbc}_0(X)$, then there exists a choice of $P$ such that $\normalfont\text{OGW}^{P}_{\beta,0,\ell} = 0$ for all $\beta, \ell$. In particular, by Remark \ref{rem:choice of P}, for any $\LL$, if there exists $OGW_{\beta,0,\ell} \neq 0$ for some $\ell$ and $\partial \beta \neq 0 \in H_1(L;\ZZ)$ then $[\LL] \neq 0 \in \normalfont\text{Cob}^{wbc}_0(X)$.
\end{cor}


\begin{example}
	Let $X$ be the Fermat quintic threefold. The real locus $X_\RR$ is a Lagrangian $\mathbb{RP}^3$ and can thus be equipped with two distinct local systems. Let $\LL = (L^+,L^-)$ be the Lagrangian brane obtained by taking $L^+$ to be $X_\RR$ with some choice of orientation and local system and $L^-$ having the opposite orientation and the other local system. As computed in \cite{PSW}, $\LL$ has non-trivial open Gromov--Witten invariants without boundary marked points and $\beta$ with $\partial \beta \neq 0 \in H_1(L;\ZZ)$. We can thus conclude that $[\LL] \neq 0 \in \normalfont\text{Cob}^{wbc}_0(X)$. This result also follows from elementary topological methods.
\end{example}

\begin{remark}\label{} 
	There are also exist cylindrical Lagrangian cobordisms (\cite{SS21}) and algebraic Lagrangian cobordisms (\cite{Cor25}). We don't study relations between induced by these kinds of cobordisms.
\end{remark}

In \cite[\S1.3.6]{ST23}, Solomon-Tukachinsky define the relative quantum connection, which extends the quantum connection on $\text{QH}^*(X,\omega)$ by including counts of discs with boundary on a Lagrangian. To make this precise, let $R$ be a graded $\Lambda$-algebra, which is used to keep track of bulk-deformations. Let $u$ be a variable of degree $2$. For a generalised Lagrangian brane $\LL$, the \emph{relative quantum connection} is a morphism
\[
\nabla^L\cl Der_\Lambda R((u)) \otimes H^*\left(Cone(\LL);R((u))\right) \rightarrow H^*(Cone(\LL);R((u))).
\]	
The results of \cite{ST23} and \cite{Hug24b} generalise directly to show that $\nabla^L$ is flat for any weakly unobstructed relatively spin Lagrangian. Our results imply moreover that $\nabla^L$ is a Lagrangian cobordism invariant.

\begin{intthm}
	The map $\Phi$ induces an isomorphism of connections 
	\begin{equation*}
		\left(H^*(Cone(\LL^-);R((u))), \nabla^{\LL^-}\right) \;\cong\; \left(H^*(Cone(\LL^+);R((u))), \nabla^{\LL^+}\right).
	\end{equation*}
\end{intthm}

In \cite{Hug24, Hug24b}, the \emph{relative cyclic open-closed map} was introduced in order to prove that the Fukaya category determines certain open Gromov--Witten invariants. These works relied on the same regularity assumptions as the papers by Solomon--Tukachinsky. Using our work, these technical assumptions can be removed.

\begin{corollary}
	For a single Lagrangian brane $L$, there exist cyclic open-closed maps \begin{equation}
		\mathcal{OC}^{\infty}\cl\, \left(HC^{\infty}(CF^*(L;R)), \nabla^{GGM} \right)\; \rightarrow\; \left(QH^{*}(X;R((u))),\nabla \right),
	\end{equation}
	and \begin{equation}
		\OC^{\infty}_L\cl\, \left(HC^\infty_*\left(CF^*(L;R), L\right),\nabla^{GGM,L}\right) \;\rightarrow\; \left(H^\star(Cone(L;R((u))))^\vee,(\nabla^{L})^\vee \right).
	\end{equation}
\end{corollary}
In order to prove that the Fukaya category determines open Gromov--Witten invariants, one would need to extend these results to the entire Fukaya category rather than just the subcategory generated by a single Lagrangian.

\subsection*{Acknowledgements}
We thank Alexia Corradini, Jeff Hicks, Jake Solomon and Sara Tukachinsky for helpful conversations. KH is grateful to Nick Sheridan for suggesting open Gromov--Witten invariants and Lagrangian cobordisms as a subject for his PhD thesis. Both authors were in residence at the Simons Center for Geometry and Physics while working on this project, and we thank the institute for its hospitality.
A.H. is supported by ERC Grant No.864919 and K.H. was partially supported by EPSRC Grant EP/W015749/1.

	\section{Review of the global Kuranishi chart construction}

Fix an embedded Lagrangian $L$ in a closed symplectic manifold $(X^{2n},\omega)$. Let $(V,\fs)$ be a relative spin structure and fix an $\omega$-tame almost complex structure $J$. Given $\beta\in H_2(X,L;\bZ)$ and $k,\ell\geq 0$, denote by $\Mbar_{k,\ell}(\beta)$ the moduli space of stable discs with boundary on $L$, $k$ boundary marked points and $\ell$ interior marked points.

\subsection{Global Kuranishi chart for moduli spaces of holomorphic discs}\label{subsec:construction-review} We briefly summarise the construction of global Kuranishi charts in \cite{HH25} for curves of genus zero with at most one boundary component.

\begin{theorem}[{\cite[Theorem~A]{HH25}}]\label{thm:gkc-equivariant-existence} The moduli space $\Mbar_{k+1,\ell}(\beta)$ admits a global Kuranishi charts $\cK = (G,\cT,\cE,\obs)$ with corners of the expected virtual dimension with the following properties.
	\begin{enumerate}[\normalfont 1),leftmargin=20pt,ref=\arabic*]
		\item\label{gkc-orientation} The orientation sheaf $\orl(\cK)$ is canonically isomorphic to the orientation sheaf of $\Mbar_{k+1,\ell}(\beta)$.
		\item\label{gkc-uniqueness} $\cK$ is unique up to the notion of equivalence.
		\item\label{gkc-cobordant} If $J'$ is another $\omega$-tame almost complex structure, then we can choose the associated global Kuranishi charts to be cobordant.
		\item \label{gkc-forgetful} $\cK$ is compatible with forgetful maps
		\item \label{gkc-symmetry} $\cK$ is invariant under cyclic permutation of the boundary marked points, and any permutation of the interior marked points
		\item \label{gkc-submersive} The evaluation map $evb_0$ is a submersion.
	\end{enumerate}
\end{theorem}

We will us the proof to recapitulate the global Kuranishi chart construction of \cite{HH25}.

\begin{proof} We first discuss the case of discs with no marked points. The input needed for the construction is an \emph{auxiliary datum} 
	\begin{equation}\label{eq:aux-dat} \alpha = (\conn^X,\cO_X(1),\cU,\lambda,r)\end{equation}
	where
	\begin{itemize}[leftmargin=20pt]
		\item $\conn^X$ is a smooth $J$-invariant connection on $TX$ that preserves $TL$,
		\item $\cO_X(1)\to X$ is an $L$-adapted polarisation, that is, a Hermitian line bundle with a Hermitian connection $\conn$ so that $\Omega := -\frac{1}{2\pi \text{i}}F^\conn$ is a $J$-invariant symplectic form and $\Omega|_L$ is exact, with $d := \Omega(\beta)$;
		\item $\cU = (\{U_i\}_{i\in I},\{\chi_i\}_{i\in I})$ is a finite open cover of (a subset of) the space of smooth stable discs to $(\bC P^d \times X,\bR P^d\times L)$ together with $\PGL_\bR(d+1)$-invariant bump functions $\chi_i$,
		\item $\lambda \cl \Mbar^{*,st}_{d}(\bC P^d,d)/S_d \to \PGL_\bR(d+1)/\normalfont\text{PO}(d+1)$ is a smooth equivariant map,
		\item $r \gg 1$ is an integer.
	\end{itemize}
	Given this, we define the complex vector bundle $E_r \to \bC P^d \times X$ by 
	$$E_r := \cc{\Hom}_\bC(T\bC^d,TX)\otimes \cO_{\bC P^d}(r)\otimes \cc{H^0(\bC P^d,\cO_{\bC P^d}(r))}$$
	with totally real subbundle
	$$E'_r := \Hom(T\bR^d,TL)\otimes \cO_{\bR P^d}(r)\otimes \cc{H^0(\bR P^d,\cO_{\bR P^d}(r))}.$$
	We can now construct the following spaces.
	\begin{enumerate}[\normalfont (1),leftmargin=20pt]
		\item A \emph{base space} $\cB \sub \Mbar_{0,0}^{\,J_0,\,d}(\bC P^d,\bR P^d)$ given by automorphism-free holomorphic maps $\varphi\cl (C,\del C)\to (\bC P^d,\bR P^d)$ representing $d[\bD]$ so that $H^1(C,\varphi^*\cO(1)) = 0$ for any irreducible component $C'$ of $C$. In particular, $\cB$ is contained in the regular locus of $\Mbar_{0,0}^{\,J_0,\,d}(\bC P^d,\bR P^d)$ and is hence a $\PGL_\bR(d+1)$-invariant smooth manifold by \cite[Proposition~2.3]{HH25}. Let $\cC\to \cB$ be its universal family, equipped with the map $\eva\cl \cC\to \bC P^d$.
		\item The family $\cZ \to \cB$ of smooth stable maps is given by 
		\begin{equation}\label{eq:family} 
			\cZ\; =\; \set{(\varphi,u)\mid \varphi\in \cB,\, u \cl (\cC|_{\varphi},\del\cC|_{\varphi})\xra{C^\infty} (X,L), u_*[\cC|_{\varphi}] =\beta,\, u^*\cO_X(1) \cong \varphi^*\cO_{\bC P^d}(1)}.
		\end{equation}
		Thus, $u$ with $(\varphi,u)\in \cZ$ is stable. Moreover, $\cZ$ admits a continuous $\PGL_\bR(d+1)$-action. 
		\item The \emph{thickening} is the open subset
		$$\cT \sub \set{(\varphi,u,\eta)\mid (\varphi,u)\in \cZ,\, \eta \in H^0(C,(\varphi,u)^*(E_r,E'_r))}$$
		where
		\begin{itemize}
			\item on the normalisation $\wt C$ of $\cC|_{\varphi}$ we have
			\begin{equation}\label{eq:perturbed-cr-equation}\delbar_J\wt u +\lspan{\eta}\g d\wt\varphi = 0 \end{equation}
			where $\wt u$ and $\wt \varphi$ are the pullbacks to $\wt C$;
			\item the perturbed Cauchy-Riemann operator 
			\begin{equation*}\label{}D_u \oplus \lspan{\cdot}\g d\wt\varphi\cl C^\infty(C,u^*(TX,TL))\oplus H^0(C,(\varphi,u)^*(E_r,E'_r))\to \Omega^{0,1}(\wt C,\wt u^*(TX,TL)) \end{equation*}
			is surjective;
			\item $H^1(C,(\varphi,u)^*(E_r,E'_r)) = 0$.
		\end{itemize}
		\item The obstruction bundle $\cE\to \cT$ is the vector bundle with fibre 
		\begin{equation}\label{}\cE_{(\varphi,u,\eta)} = H^0(C,(\varphi,u)^*(E_r,E'_r))\oplus \fp(d+1)\end{equation}
		where $\fp(d+1)$ is the Lie algebra of positive definite symmetric matrices of determinant $1$.
	\end{enumerate}
	The obstruction section $\obs \cl \cT\to \cE$ is given by 
	\begin{equation}\label{eq:obstruction-section}\obs(\varphi,u,\eta) = (\eta,\lambda_\cU(\varphi,u)), \end{equation} where $\lambda_\cU$, constructed in \cite[\textsection2.6]{HH25}, is a $\text{PO}(d+1)$-equivariant map associated to $\cU$ and $\lambda$.\par 
	We call $\alpha$ \emph{unobstructed} if the canonical map $\obs\inv(0)\to \Mbar_{0,0}(\beta)$ is surjective. The existence of unobstructed auxiliary data was proven in \cite[\textsection2]{HH25}, where it was also shown that $(\text{PO}(d+1),\cT/\cB,\cE,\obs)$ is a rel--$C^\infty$ global Kuranishi chart for $\Mbar_{0,0}(\beta)$. The existence of an induced smooth structures follows from \cite[Theorem~C.29]{HH25}.
	 
	To recover the claim with marked points, let $\cB_{k,\ell}\sub \Mbar_{k,\ell}^{\,J_0,d}(\bC P^d,\bR P^d)$ be the preimage of $\cB$ under the map forgetting all marked points. Writing $\cK$ for the above constructed global Kuranishi chart, the pullback 
	$$\cK_{k,\ell} = \cK\times_\cB\cB_{k,\ell} := (\text{PO}(d+1),\cT\times_\cB\cB_{k,\ell},\cT\times_\cB\cB_{k,\ell},\obs\times\ide)$$
	is a global Kuranishi chart for $\Mbar_{k,\ell}(\beta)$.
\end{proof}

\subsection{Recursive structure of the boundary}\label{subsec:boundary-strata}
We restate \cite[Theorem~3.5]{HH25} for the special case discs. Thus, there are two types of boundary degeneration: disc bubbling and boundary collapse. The latter happens exactly if the boundary carries no marked points. In either case, we can obtain a global Kuranishi chart for the (normalised) boundary stratum by pulling back the global Kuranishi chart $\cK_\alpha$ constructed in the previous subsection. To make this concrete, suppose $k = k_0+k_1$ and $1 \leq i \leq k_1$ and let $I \sqcup J = \{1,\dots,\ell\}$ be a decomposition. The associated clutching map is
\begin{equation}\vartheta_{L,i}\cl \Mbar_{k_0+1,I}(\beta_0)\times_L \Mbar_{k_1+1,J}(\beta_0)\to \Mbar_{k+1,\ell}(\beta)\end{equation}
and its image is a virtual boundary stratum of $\Mbar_{k,\ell}(\beta)$. Note that in \cite[\S 3]{HH25} we denote this map by $\varphi$ instead. Set $d_i := \Omega(\beta_i)$ and let
\begin{equation}\label{eq:stratum-base}\cB_{\varphi,i} := \varphi_{\bR P^d}\inv(\cB_{k,\ell}).\end{equation}
It lies in the automorphism-free regular locus of $\Mbar_{k_0+1,I}^{\,J_0,d_0}(\bC P^d,\bR P^d)\times_{\bR P^d} \Mbar_{k_1+1,J}^{\,J_1,\,d_1}(\bC P^d,\bR P^d)$ and is thus a smooth manifold with corners admitting a smooth $\cG$ action. 
We recall the convention that the boundary marked points are ordered anti-clockwise. Thus, we have to reorder marked points as shown in \cite[Assumption~3.3]{HH25}. We can now let 
$$\cT_{\varphi,i}^{\beta_0,\beta_1}\sub \cT_\alpha\times_{\cB_{k,\ell}} \cB_{\varphi,i}$$
be the open (and closed) subset of maps $u \cl C_0 \vee C_1 \to X$ so that $u_*[C_i] = \beta_i$. Writing 
$\wt \varphi_i \cl \cT_{\varphi,i}^{\beta_0,\beta_1}\to \cT_\alpha$ for the canonical projection, we define the global Kuranishi chart 
$$\cK_{\varphi,i}^{\beta_0,\beta_1} :=  (G,\cT^{\beta_0,\beta_1}_{\varphi,i},\wt\varphi_i^*\cE_\alpha,\wt\varphi_i^*\obs).$$
By \cite[Theorem~3.5(a)]{HH25}, this is a global Kuranishi chart with corners for the fibre product $\Mbar_{k_0+1,I}(\beta_0)\times_L \Mbar_{k_1+1,I}(\beta_1)$, which is equivalent to the fibre product global Kuranishi chart with orientation sign 
\begin{equation}\label{h3-sign} (-1)^{n+\ell_1(\ell_0-i+1) + i-1}.\end{equation}
Similarly, by \cite[Theorem~3.5(c)]{HH25}, we can construct a global Kuranishi chart $\cK_{\rho}$ from $\cK^\beta$ for $L\times_X\Mbar_{\emst,\ell+1}(\beta)$, which is equivalent to $L\times_X\cK_{0,\ell+1,\alpha}^\beta$ with orientation sign
\begin{equation}\label{e-sign} (-1)^{n+w_\fs(\beta)}.\end{equation}

\subsection{Cubical cobordisms}\label{subsec:cubical-cobordisms}
From the proof of the equivalence statement in \cite[Theorem~3.5]{HH25}, we obtain in particular that (a stabilisation) of the fibre product global Kuranishi chart of $\Mbar_{k_0+1,I}(\beta_0)\times_L \Mbar_{k_1+1,I}(\beta_1)$ (and $L\times_X\cK_{\emst,\ell+1,\alpha}^\beta$) is cobordant to (a stabilisation) of the `pullback chart'. We will use this fact to `interpolate' between different choices of auxiliary data (leading to algebraic operations that are per se not related) and insert these interpolations into the operations we define in \textsection\ref{subsec:q-operations}. The details can be found in \cite[\textsection 4]{HH25}, especially Theorem~4.1; we just describe the idea behind the construction as it might be more enlightening. 
\noindent Define the sets
\begin{align*}
	\cA_o &\coloneqq \set{(\beta,k,\ell)\mid\beta \in H_2(X,L;\bZ),\,\omega(\beta),k,\ell \geq 0}\\
	\cA_c &\coloneqq \set{(\beta,\ell)\mid\beta \in H_2(X;\bZ),\,\omega(\beta),\ell \geq 0},
\end{align*}
and set $\cA := \cA_o \sqcup \cA_c$. For each $\mfa \in \cA$ let $\Mbar_{\mathsf{a}}^{\,J}(X,L)$ denote the moduli space of $J$ holomorphic stable maps (from either a sphere or disc) of type $\mfa$. By a \emph{stable map tree} $\Gamma$ we mean the combinatorial data underlying the domain of a (possibly open) stable map; see \cite[Definition~4.16]{HH25} for the precise definition. In contrast with \cite{HH25}, we do not explicitly divide marked points into incoming or outgoing ones; the (boundary or interior) marked point labelled by $0$ should be thought of as outgoing. Each vertex $v \in V(\Gamma)$ comes with the data $\mfa_{v} \in \cA$. Denote by 
$$\Mbar_{\Gamma}^{\,J}(X,L)\sub \p{v \in V(\Gamma)}{\Mbar_{\mathsf{a}_v}^{\,J}(X,L)}$$ \noindent
the associated fibre product, where each boundary edge gives a fibre product over $L$, and an interior edge a fibre product over $X$. It admits an embedding 
\begin{equation}
	\label{eq:embedding-of-fibre-product-0}\Mbar_{\Gamma}^{\,J}(X,L)/\Aut(\Gamma)\to \Mbar_{\mathsf{a}}^{\,J}(X,L)
\end{equation}
onto a (virtual) boundary or divisor stratum $\del_\Gamma\Mbar_{\mathsf{a}}^{\,J}(X,L)$ of $\Mbar_{\mathsf{a}}^{\,J}(X,L)$. We consider only the morphisms $\Gamma\to \Gamma'$ given by contraction of edges (and corresponding merging of decorations) and the contraction of the boundary of a disc to a point. Such a morphism is \emph{elementary} if it contracts a single edge or only turns a single interior marked point into a boundary node.\par 
If we choose an unobstructed auxiliary datum $\alpha_\mathsf{a}$ for each $\mathsf{a} \in \cA$ with $k_{\mathsf{a}} = \ell_{\mathsf{a}} = 0$ and the trivial auxiliary datum for $\mathsf{a} = (0,k,\ell)$, Theorem~\ref{thm:gkc-equivariant-existence} yields a global Kuranishi chart $\cK_{\mathsf{a}} \coloneqq \cK_{\alpha_{\mathsf{a}}}$ for each $\mathsf{a} \in \cA$. Note that we can use the same unobstructed auxiliary data for $(\beta,k,\ell)$ and $(\beta,0,0)$, respectively $(\beta,\ell)$ and $(\beta,0)$. By taking the fibre product chart, we obtain then a global Kuranishi chart $\cK_\Gamma$ for each stable map tree $\Gamma$. 

\begin{notation*} If $\cK'$ is obtained from $\cK$ by stabilisation and group enlargement, we write $\cK' \stackrel{\sim\,}{\longrightarrow} \cK$, possibly with a sign in front to indicate the difference in orientation if the charts are oriented.
\end{notation*}

\begin{proposition}[{\cite[Theorem~4.1]{HH25}}]\label{prop:cubical-cobordisms}
	For each stable map tree $\Gamma$, there exists an orientable global Kuranishi chart $\cKc_{\Gamma}$ for $I^{E(\Gamma)} \times \Mbar_\Gamma(X,L)$ so that following holds.
	\begin{enumerate}[\normalfont i),leftmargin=20pt,ref=\roman*]
		\item If $\Gamma = \mathsf{a}$, i.e., it has no edges, then $\cKc_\Gamma = \cK_{\mfa}$.
		\item\label{cubical-0-boundary} for any edge $e$ of $\Gamma$ we have a local embedding (up to stabilisation and group enlargement)
		\begin{equation}\label{eq:restricted-e-0}
			\cKc_{\Gamma}|_{\{t_e = 0\}} \;\stackrel{\sim\,}{\longrightarrow}\; (\pm 1)\,\del_\Gamma\cKc_{\Gamma_e}
		\end{equation}
		where $\Gamma_e$ is the graph obtained from $\Gamma$ by contracting the edge $e$. This local embedding has degree \[d_1(\Gamma,\Gamma_e) = |\{\phi\in \Aut(\Gamma)\mid f\g\phi = f\},\] where $f\cl \Gamma\to \Gamma_e$ is the underlying contraction of graphs.
		\item\label{cubical-1-boundary}
		The restriction to the other boundary face satisfies
		\begin{equation}
			\label{eq:restricted-e-1} \cKc_{\Gamma}|_{\{t_e = 1\}} \;\stackrel{\sim\,}{\longrightarrow}\; (\pm 1)\,\cKc_{\Gamma_1}\times_{Y}\cKc_{\Gamma_2},
		\end{equation}
		where $\Gamma_1*_e \Gamma_2 = \Gamma$ and $Y = X$ or $L$, depending on whether $e$ is an interior or a boundary edge, respectively
		\begin{equation}
			\label{eq:restricted-e-2} \cKc_{\Gamma}|_{\{t_e = 1\}} \;\stackrel{\sim\,}{\longrightarrow}\; (\pm 1)\,L\times_X\cKc_{\Gamma'}
		\end{equation}
	in the case of boundary collapse.
	\item\label{forgetful-cobordisms} If $\Gamma$ is obtained from $\Gamma'$ by forgetting marked points, the forgetful map $\Mbar_{\Gamma'}(X,L) \to \Mbar_{\Gamma}(X,L)$ lifts to a smooth map $\cKc_{\Gamma'}\to \cKc_\Gamma$.
	\end{enumerate}
\end{proposition}

\noindent
	We call each $\cKc_\Gamma$ a \emph{cubical cobordism} and the whole collection a \emph{system of cubical cobordisms}.

In particular, by considering the thickening as an orbifold, we obtain maps of orbibundles 
\begin{equation*}\label{} \Pi_0\cl \cEc_\Gamma|_{\{t_e = 0\}}\to \cEc_{\Gamma_e} \qquad \qquad\Pi_1\cl \cEc_\Gamma|_{\{t_e = 1\}}\to \cEc_{\Gamma_1}\times \cEc_{\Gamma_2}|_{\cTc_{\Gamma_1}\times_Y \cTc_{\Gamma_2}} \end{equation*} 
and analogously for the case of boundary collapse. Given a choice of Thom form $\eta'_{\mfa}$ for each $\mfa$, their wedge products (and restrictions) yield Thom form $\eta'_\Gamma$ of $\cK_\Gamma$ for each stable map tree $\Gamma$. By \cite{HH25}, we can extend this collection $\{\eta'_\Gamma\}_\Gamma$ to a \emph{Thom system} for $\{\cKc_\Gamma\}_\Gamma$, that is a Thom form $\eta_\Gamma$ for each $\cKc_\Gamma$ so that 
\begin{equation}\label{eq:compatible-0} 
	(\Pi_0)_*(\eta_\Gamma|_{\{t_e = 0\}}) = \eta_{\Gamma_e}|_{\del_\Gamma}\end{equation}
for each edge $e$ and 
\begin{equation}\label{eq:compatible-1} 
	(\Pi_1)_*(\eta_\Gamma|_{\{t_e = 1\}}) = \eta_{\Gamma_1}\times_Y \eta_{\Gamma_2}\end{equation}
if $e$ is an edge and 
\begin{equation}\label{eq:compatible-12} 
	(\Pi_1)_*(\eta_\Gamma|_{\{t_e = 1\}}) = \eta_{\Gamma'}|_{L\times_X\cTc_{\Gamma'}}\end{equation}
if $e$ indicates a boundary collapse. Such Thom systems exists by \cite[Theorem~4.32]{HH25}; moreover since we have at most one outgoing marked point, said theorem together with \cite[Proposition~4.12]{HH25} also yields the following property.

\begin{lemma}\label{lem:thom-form-and-forgetful}
	We can choose $\eta_\Gamma$ so that for $\Gamma$ obtained from $\Gamma'$ by forgetting marked points, the Thom form $\eta_{\Gamma'}$ is the pullback of $\eta_\Gamma$. In particular, each $\eta_\Gamma$ is invariant under permutations of marked points.
\end{lemma}

	\section{Open Gromov--Witten invariants in genus zero}\label{sec:q-operations}
For any smooth manifold $Y$, let $\Omega^*(Y;\Lambda) \coloneqq \Omega^*(Y)\otimes \Lambda$ denote the completed tensor product. Throughout this section, $(X,\omega)$ is a closed symplectic manifold of real dimension $2n$ and $L\sub X$ is a nonempty embedded Lagrangian equipped with a relative spin structure $(V,\fs)$.\par
\noindent Consider the Novikov rings
\begin{align*}
	\Lambda \,&\coloneqq  \set{  \sum_{i = 0}^{\infty} a_iQ^{\beta_i} \,\big|\, a_i \in \mathbb{R}, \; \beta_i \in H_2(X, L), \; \lim_{i\to \infty } \omega(\beta_i) = \infty }\\
	\Lambda_c &\coloneqq  \set{  \sum_{i = 0}^{\infty} a_iQ^{\beta_i} \,\big|\,  a_i \in \mathbb{R}, \; \beta_i \in H_2(X), \; \lim_{i\to \infty } \omega(\beta_i) = \infty}
\end{align*}
These are graded rings with $|Q^\beta| = \mu(\beta)$ and $|Q^{\beta}| = 2c_1(\beta)$ respectively.
We write $|\al|$ for the degree of a differential form $\alpha$ and set $|\alpha|' := |\alpha|-1$.

\subsection{Construction of the $\fq$-operations}\label{subsec:q-operations} Suppose now that we have the data of a system $\{\cKc_\Gamma\}_\Gamma$ of cubical cobordisms as defined in the previous subsection and an associated Thom system $\{\eta_\Gamma\}$.  We assume that if $\Gamma$ has energy zero, then $\cKc_\Gamma$ is the trivial global Kuranishi chart $I^{E(\Gamma)}\times \Mbar_\Gamma$.
In what follows, it might be helpful to think of the `total chart' for $\Mbar_\mfa(X,L)$ as the \emph{spiky chart} $\cK_{\mathsf{a}}^+$ given by
	\begin{equation*}
		\label{}\cK_{\mathsf{a}}^+ \coloneqq \djun{\substack{\Gamma \rightarrow \mfa}}\cKc_{\Gamma}.
	\end{equation*}
For any $\mfa$, let $\del(\mfa)$ denote the set of stable map trees $\Gamma$ which contract to $\mfa$ and are such that $\cKc_\Gamma$ has the same virtual dimension as $\cK_\mfa$. Then $\Gamma \in \del(\mfa)$ if and only if all edges of $\Gamma$ correspond to codimension $1$ phenomena, that is, disc bubbling, or boundary contracting.
For $k,\ell \geq 0$ and $\beta \in H_2(X,L)$ with $(k,\ell,\beta) \notin \lbrace (1,0, \beta_0), (0,0,\beta_0) \rbrace$, 
define the operations
\begin{equation*}
	\fq_{k,\ell}^\beta\cl \Omega^*(L)^{\otimes k} \otimes \Omega^*(X)^{\otimes \ell} \rightarrow \Omega^*(L)
\end{equation*} 
as follows. Let $\mfa = (k+1,\ell,\beta)$ and set
\begin{equation}
	\label{eq: q operation defi}
	\mathfrak{q}_{k,\ell}^{\beta}(\alpha_1, \dots,\alpha_k; \gamma_1,\dots,\gamma_\ell) \coloneqq  (-1)^{\epsilon(\alpha)} \sum_{\Gamma \in\del\mfa}(evb_0^\Gamma)_* \bigg{(}  \bigwedge_{j=1}^\ell (evi^\Gamma_j)^* \gamma_j \wedge \bigwedge_{i=1}^k (evb_i^\Gamma)^* \alpha_j \wedge \obs_\Gamma^*\eta_{\Gamma} \bigg{)}, 
\end{equation}
where $\epsilon(\alpha) \coloneqq  1 + \sum_{j=1}^{k} j|\alpha_j|'$. The special (unstable) cases are 
\begin{equation*}
	\mathfrak{q}_{1,0}^{\beta_0} \coloneqq  d\alpha\qquad\quad \text{and}\quad \qquad 
	\mathfrak{q}_{0,0}^{\beta_0} \coloneqq  0.
\end{equation*}
Similarly, we define disc counts with only interior marked points
\begin{equation*}
	\mathfrak{q}_{-1,\ell}^{\beta}\cl \Omega^*(X)^{\otimes \ell } \rightarrow \mathbb{R}.
\end{equation*}
by setting  
\begin{equation}\label{eq:q-with-no-boundary-markings}
	\mathfrak{q}_{-1,\ell}^{\beta}(\gamma_1, \dots,\gamma_\ell) \coloneqq  \sum_{\Gamma \in \del (0,\ell,\beta)}c_\Gamma \int_{\cTc_\Gamma}  \bigwedge_{j=1}^\ell (evi^\Gamma_j)^* \gamma_j \wedge \obs_\Gamma^*\eta_{\Gamma}
\end{equation}
for $(\ell,\beta) \notin \lbrace (0, \beta_0), (0,\beta_0) \rbrace$, $c_\Gamma = \frac{|\Aut(\Gamma)|}{|\Aut(\Gamma^{ps})|}$ with $\Gamma^{ps}$ being the prestable graph obtained from $\Gamma$ by forgetting the degrees. This coefficient is $1$ for the graphs appearing in the definition of $\fq_{k,\ell}$ with $k \ge 0$. We have the special cases
\begin{equation*}
	\mathfrak{q}_{-1,1}^{\beta_0} \coloneqq  0 \qquad\quad \text{and}\quad \qquad  \mathfrak{q}_{-1,0}^{\beta_0} \coloneqq  0.
\end{equation*}

\noindent
For $\beta \in H_2(X)$, the operations associated to closed curves are 
\begin{equation*}
	\mathfrak{q}_{\emptyset, \ell}^{\beta}\cl \Omega^*(X)^{\otimes \ell} \rightarrow \Omega^*(X)
\end{equation*}
defined by 
\begin{equation*}
	\label{eq: closed q operation}
	\mathfrak{q}_{\emptyset, \ell}^{\beta}(\gamma_1, \dots, \gamma_\ell) \coloneqq  (-1)^{w_{\mathfrak{s}}(\beta)} \sum_{\Gamma \in \del(\ell, \beta)} (\eva_0^\Gamma)_*\lbr{ \bigwedge_{j=1}^\ell (ev_j^\Gamma)^* \gamma_j\wedge \obs_\Gamma^*\eta_{\Gamma}},
\end{equation*}
with the two special special cases \begin{equation*}
	\mathfrak{q}_{\emptyset, 1}^{0} \coloneqq  0 \qquad\quad \text{and}\quad \qquad \mathfrak{q}_{\emptyset, 0}^0 \coloneqq  0.
\end{equation*}
For all of the above $\mathfrak{q}_{*}^\beta$ operations, set $
	\mathfrak{q}_{*} = \sum_\beta T^{\beta} \mathfrak{q}^\beta_{*}$.
	
\begin{remark}\label{} The extra coefficient in~\eqref{eq: q operation defi} and~\eqref{eq:q-with-no-boundary-markings} is due to the fact that the clutching map from the fibre product to the boundary stratum can have degree greater than $1$. It does not appear in~\eqref{eq: closed q operation} since the only unstable vertices on the graphs $\Gamma$ are discs, and these don't appear there. The reason underlying this exact choice of coefficient is discussed in \cite[\textsection4.3]{HH25}.
\end{remark}

\subsection{Properties}\label{sec:props-of-q-ops}
In this subsection we show that these $\q$-operations satisfy the same properties as those defined in \cite{ST16}. Let 
$$\langle \alpha_1, \alpha_2 \rangle_L \coloneqq  (-1)^{|\alpha_2|}\int_L \alpha_1 \wedge \alpha_2$$ 
be the signed Poincar\'e pairing on $L$. Furthermore, let $R$ and $Q$ be unital, graded $\Lambda_L$ and $\Lambda_c$-algebras (these will keep track of bulk-deformations, see \S\ref{section: bulk-deformations and bounding cochains}) and extend the $\q$-operations (graded) linearly to $\Omega^*(L;R) \coloneqq  \Omega^*(L) \otimes R$ and $\Omega^*(X;Q) \coloneqq  \Omega^*(X)\otimes Q$.

\begin{proposition}[Unit]
	\label{unit property}
	For $f \in \Omega^0(L;R)$, $\alpha_1, \dots, \alpha_{k-1} \in \Omega^*(L;R)$ and $\gamma \in \Omega^*(X;Q)^{\otimes \ell}$, we have
	\begin{equation*}
		\mathfrak{q}^\beta_{k,\ell}(\alpha_1, \dots, \alpha_{i-1},f,\alpha_i, \dots, \alpha_{k-1}; \gamma) =
		\begin{cases}
			df  &(k,\ell,\beta) = (1,0,\beta_0)\\
			(-1)^{|f|}f\alpha_1 &(k,\ell,\beta) = (2,0,\beta_0), \; i=1\\
			(-1)^{|\alpha_1||f|'}f\alpha_1 & (k,\ell,\beta) = (2,0,\beta_0), \; i=2\\
			0 & \text{otherwise}
	\end{cases}\end{equation*}
\end{proposition}

\begin{proof}
	As our $\q$ operations are compatible with forgetting boundary marked points, the same reasoning as in the proof of \cite[Proposition~3.2]{ST16} shows the result for $(k,\ell,\beta) \neq (2,0,\beta_0)$. In the case of $(2,0,\beta_0)$, by construction we have $\cT_{3,0}(\beta) = \Mbar_{3,0}(\beta)$, whence the result follows by the proof of \cite[Proposition~3.2]{ST16}.
\end{proof}

\begin{proposition}[Cyclic symmetry]
	\label{cyclic symmetry}
	For any $\alpha_1, \dots, \alpha_{k+1}\in \Omega^*(L;R)$ and $\gamma \in \Omega^*(X;Q)^{\otimes \ell}$ we have
	\begin{equation*}
		\langle \mathfrak{q}_{k,\ell}(\alpha_1, \dots, \alpha_k; \gamma),\alpha_{k+1} \rangle_L = (-1)^{|\alpha_{k+1}|'(|\alpha_1|' + \dots |\alpha_k|')} \langle  \mathfrak{q}_{k,\ell}(\alpha_{k+1}, \alpha_1, \dots, \alpha_{k-1}; \gamma),\alpha_{k} \rangle_L
	\end{equation*}
\end{proposition}
\begin{proof}
	The proof is analogous to the proof of \cite[Proposition~3.3]{ST16} applied to each $\cKc_\Gamma$ separately. The extra property we need is the last assertion of Lemma~\ref{lem:thom-form-and-forgetful}, ie, that we can choose Thom forms to be invariant under permuting the boundary marked points. 
\end{proof}

\begin{proposition}[Degree]
	\label{degree property}
	For any $\alpha \in \Omega^*(L;R)^{\otimes k}$ and $\gamma \in \Omega^*(X;Q)^{\otimes\ell}$, 
	\begin{align*}
		|\mathfrak{q}^\beta_{k,\ell}(\alpha; \gamma)| &= 2 + |\alpha|' -\mu(\beta) + \sum_{j=1}^\ell (|\gamma_j| - 2)  \\
		&\equiv |\alpha|' + \sum_{j=1}^\ell |\gamma_j| \; (\normalfont\text{mod } 2),
	\end{align*}
	where $|\alpha|' = \sum_{i =1}^{k}(|\alpha_i|-1)$.
\end{proposition}

\begin{proof}
	This follows from the same proof as \cite[Proposition~3.5]{ST16}, noting that the Thom form has the same rank as the obstruction bundle.
\end{proof}


\begin{proposition}[Symmetry]\label{symmetry}
	If $k \geq -1$ or  $k = \emst$, then for any permutation $\sigma \in S_\ell$ we have
	\begin{equation*}
		\fq_{k, \ell}(\al_1, \dots, \al_k;\ga_1, \dots, \ga_\ell) = (-1)^{\epsilon_\sigma(\gamma)}\fq_{k, \ell}(\al_1, \dots, \al_k;\ga_{\sigma(1)}, \dots, \ga_{\sigma(\ell)}),
	\end{equation*}
	where $$\epsilon_\sigma(\ga) \coloneqq  \s{\substack{i<j\\\sigma^{-1}(i) > \sigma^{-1}(j)}}{|\gamma_i||\gamma_j|}.$$
\end{proposition}

\begin{proof}
	This follows by the same arguments as in the proof of \cite[Proposition~3.6]{ST16} combined with the last assertion of Lemma~\ref{lem:thom-form-and-forgetful}.
\end{proof}

\begin{proposition}[Fundamental class]
	\label{fundamental class property}
	For $k \geq -1$ and $\ell \geq 0$, we have
	\begin{equation*}
		\mathfrak{q}_{k,\ell}^{\beta}(\alpha; 1, \ga') = 
		\begin{cases}
			-1,\quad &(k,\ell,\beta) = (0,1,\beta_0),\\
			0, \quad& \;\text{otherwise}
	\end{cases}\end{equation*}
	for any $\alpha\in \Omega^*(L;R)^{\otimes k}$ and $\ga'\in \Omega^*(X;Q)^{\otimes \ell-1}$.
\end{proposition}

\begin{proof}
	The proof is to similar the proof in \cite[Proposition~3.7]{ST16}, using the compatibility of the cubical cobordisms and the Thom forms with forgetting interior marked points, that is, Proposition~\ref{prop:cubical-cobordisms}\eqref{forgetful-cobordisms} and Lemma~\ref{lem:thom-form-and-forgetful}.
\end{proof}

\begin{proposition}[Energy zero]\label{energy zero property}
	For $k \geq -1$, we have
	\begin{equation*}
		\mathfrak{q}_{k,\ell}^{\beta_0}(\alpha; \ga) = \begin{cases}
			d\alpha_1, \quad&(k,\ell) = (1,0)\\
			(-1)^{|\alpha_1|}\alpha_1 \wedge \alpha_2, \quad&(k,\ell) = (2,0),\\
			-\gamma_1|_L,\quad &(k,\ell) = (0,1),\\
			0, &\text{otherwise}.
		\end{cases}
	\end{equation*}\qed
\end{proposition}

\begin{proof}
	Since we use the trivial global Kuranishi chart for the relevant moduli spaces, this is almost exactly \cite[Proposition~3.8]{ST16}. However, we are summing over operations associated to moduli space of stable trees, so we have to show these pushforwards vanish whenever the stable tree $\Gamma$ has edges. This is due to the property of the pushforward: if $\Gamma$ has edges, the pushforward map $L\times\Mbar_\Gamma\times I^{E(\Gamma)}\to L$ vanishes on forms pulled back from $L$ because it amounts to integrating a differential form of degree $0$ over the positive-dimensional manifold $\Mbar_\Gamma\times I^{E(\Gamma)}$.
\end{proof}

\begin{proposition}[Divisor]
	\label{divisor property}
	If $\gamma_1 \in \Omega^2(X,L)$ is closed, then, for $k \geq -1$ and $\ell \geq 1$, we have \begin{equation*}
		\mathfrak{q}^\beta_{k,\ell}(\alpha;\gamma) = \left(\int_\beta \gamma_1\right)\cdot \mathfrak{q}^\beta_{k,\ell-1}(\alpha;\ga_2, \dots, \ga_\ell).
	\end{equation*}
\end{proposition}
\begin{proof}
	The proof is similar to that of \cite[Proposition~3.9]{ST16}. Let $\cT$ be the thickening for $\Mbar_{0,0}(\beta)$. By Theorem \ref{thm:gkc-equivariant-existence}, respectively Proposition~\ref{prop:cubical-cobordisms}\eqref{forgetful-cobordisms}, 
	$$\cTc_{\Gamma} = \cTc_{\Gamma_0} \times_{\cBc_{\Gamma_0}} \cBc_{\Gamma},$$
	where $\Gamma_0$ is the stable map tree obtained from $\Gamma$ by forgetting all marked points. In particular, by Lemma~\ref{lem:thom-form-and-forgetful}, we have  $\pi^*\obs^*\eta_{\Gamma_0} = \obs^*\eta_{\Gamma}$, where $\Gamma_{-1}$ is the graph obtained from $\Gamma$ by forgetting an interior marked point and and $\pi \cl \cTc_{\Gamma} \rightarrow \cTc_{\Gamma_{-1}}$ is the forgetful map. As the forgetful maps intertwine the evaluation maps, we have the following equality of currents 
	\begin{align*}
		evi^{\ell}_j = evi^{\ell-1}_j \circ \pi \text{ for } j \geq 2.
	\end{align*}
	on the thickenings. Thus, 
	\begin{equation*}
		(evb^{\Gamma}_0)_*\lbr{(evi_1^{\Gamma})^*\gamma_1 \wedge \pi^*\xi \wedge \obs_\Gamma^*\eta_{\Gamma}} = (evb^{\Gamma_{-1}}_0)_*\lbr{\pi_*(evi_1^\Gamma)^*\gamma_1 \wedge \xi \wedge \obs_{\Gamma_{-1}}^*\eta_{\Gamma_{-1}}}.
	\end{equation*}
	This shows that we need to compute $\pi_*(evi_1^\ell)^*\gamma_1$. By applying (the proofs of) Lemmas~3.10 and 3.11 in \cite{ST16}, we find that it suffices to compute the value of this on a single point $x$ which is regular for the forgetful map $\pi$. This forgetful map is the pullback of the forgetful map $p \cl \cBc_{\Gamma} \rightarrow \cB_{\Gamma_{-1}}$ on the level of base spaces. Thus, we have a canonical isomorphism $\pi^{-1}(x) \cong p^{-1}(x)$. The proof of Lemma~3.11 in \cite{ST16} now yields the equality $$\pi_*(evi_1^\Gamma)^*\gamma_1 = \int_\beta \gamma_1$$ of currents. As the space of differential forms embeds into the space of currents by Poincar\'e duality, the result follows.
\end{proof}

\begin{proposition}[Top degree]
	\label{top degree property}
	If $\delta_n$ denotes the (cohomological) degree $n$ part of $\delta \in \Omega^*(L;R)$, then $$\mathfrak{q}^\beta_{k,\ell}(\alpha;\gamma)_n = 0$$ for $(k,\ell,\beta) \notin \{ (1,0,\beta_0), (0,1,\beta_0), (2,0,\beta_0)  \}$.
\end{proposition}

\begin{proof}
	The proof in \cite[Proposition~3.3]{ST16} carries over as the system of cubical cobordisms we construct are compatible with the forgetting boundary marked points, Proposition~\ref{prop:cubical-cobordisms}\eqref{forgetful-cobordisms} and Lemma~\ref{lem:thom-form-and-forgetful}.
\end{proof}

We now turn to the structural relations the $\fq$-operations satisfy. For this, let  $\gamma \in \Omega^*(X;Q)^{\otimes\ell}$ and $I \sqcup J = \{ 1, \dots, \ell\}$ be a partition. Then, we define $\sign^\gamma(I,J)$ by the equation 
\begin{equation*}
	\bigwedge_{i\in I} \gamma_i \wedge \bigwedge_{j\in J} \gamma_j = (-1)^{\sign^\gamma(I,J)}\bigwedge_{s\in [\ell]} \gamma_s, 
\end{equation*}
where $\sign^\gamma(I,J) = \sum_{\substack{i \in I, j \in J\\j<i}} |\gamma_i||\gamma_j|$.

\begin{proposition}[Structure equation for $\q_{k,\ell}$]\label{prop:q structure}
	For any $k,\ell \geq 0$ as well as $\alpha \in \Omega^*(L;R)^{\otimes k}$ and $\gamma \in \Omega^*(X;Q)^{\otimes\ell}$, we have
	\begin{align*}
		0 &= \sum_{\substack{ P \in S_3[k] \\ (2:3) = \{ j \}}}  (-1)^{|\gamma^{(1:3)}| + 1}\mathfrak{q}_{k,\ell}(\alpha;\gamma^{(1:3)} \otimes d\gamma_j \otimes \gamma^{(3:3)})\\
		&\quad + \sum_{\substack{ P \in S_3[k] \\ I \sqcup J = [\ell]}} (-1)^{i(\alpha, \gamma, P, I)} \mathfrak{q}_{k_1+1+k_3,|I|}(\alpha^{(1:3)}\otimes \mathfrak{q}_{k_2,|J|}(\alpha^{(2:3)};\gamma^{J})\otimes \alpha^{(3:3)};\gamma^I),
	\end{align*}
	where 
	\begin{equation*}
		i(\alpha,\gamma, P, I) \coloneqq  (|\gamma_J| + 1)\epsilon_1 + |\gamma_I| + sign^\gamma(I,J).
	\end{equation*}
\end{proposition}

\begin{proof}
	The proof combines the computation of \cite[Proposition~2.4]{ST16} with \cite[Lemma~5.16]{HH25}. Recall that $\fq_{k,\ell}^\beta$ is the sum of operations $\fq_{\Gamma}$ as in Equation~\eqref{eq: q operation defi}, where $\Gamma$ ranges over the stable map trees where any vertex that is unstable after forgetting the degree is a disc. By Stokes' Theorem,
	\[d\fq_\Gamma(\alpha;\gamma) = (-1)^{|\gamma|}\fq_{\Gamma}(d\alpha;\gamma)+ \fq_{\Gamma}(\alpha;d\gamma) + (\eva^\Gamma_0|_{\del\cTc_\Gamma})_*\lbr{\evai^*\gamma\wedge \evab^*\alpha\wedge\obs^*\eta_\Gamma|_{\del\cTc_\Gamma}}.\]
	The operations given by this last term are of three types, depending on the boundary strata of $\cTc_\Gamma$ as discussed in \S\ref{subsec:cubical-cobordisms}. Concretely, writing $\Gamma_e$ for the tree obtained from $\Gamma$ by contracting the edge $e$ and writing $\Gamma'_e$ and $\Gamma''_e$ for the trees obtained by cutting the edge $e$, we obtain (up to sign) the operations
	\begin{enumerate}[label=\roman*),leftmargin=20pt,ref=\roman*]
		\item\label{first} $\fq_{\Gamma_e}(\alpha;\gamma)$ for $t_e = 0$,
		\item\label{second} $\fq_{\Gamma_e'}(\alpha',\fq_{\Gamma'_e}(\alpha'';\gamma'');\gamma')$ for $t_e = 1$,
		\item\label{third} $\fq_{\Gamma'}(\alpha;\gamma)$ for the boundary stratum $\del_{\Gamma'}\cTc_{\Gamma'}$, where $\Gamma'$ admits an edge $e'$ so that $\Gamma = \Gamma'_{e'}$. 
	\end{enumerate}
	Since we sum over trees with only boundary edges, these are all the operations we obtain. It is shown in \cite[Lemma~3.32]{HH26} that when summing over all $\Gamma \in \del(\beta,k,\ell)$, the operations of~\eqref{first} cancel with those of~\eqref{third}. Since we orient $\cTc_\Gamma$ in the same way we orient $\cT_{k,\ell}^\beta$, the orientation signs in front of~\eqref{second} are exactly as in the proof of \cite[Proposition~2.4]{ST16}. Thus, the computation there yields the claim.
\end{proof}

\begin{corollary}
The operations $\{\fm_k \coloneqq  \q_{k,0}\}_{k \geq 0}$ define a curved cyclic unital $A_\infty$ structure on $\Omega^*(L;\Lambda_L)$.
\end{corollary}

\noindent
Similar to \cite[Proposition~2.5]{ST16} there also exists a structure equation for the $\q_{-1}$ operation, using Equation~\eqref{e-sign} for the required sign computation.

\begin{proposition}[Structure equation for $\q_{-1,\ell}$]
	\label{prop:q-1-structure}
	For any $\ga \in \Omega^*(X;Q)^{\otimes\ell}$ we have
	\begin{align*}
		0 &= \sum_{\substack{ P \in S_3[k] \\ (2:3) = \{ j \}}}  (-1)^{|\gamma^{(1:3)}| + 1}\mathfrak{q}_{-1,\ell}(\gamma^{(1:3)} \otimes d\gamma_j \otimes \gamma^{(3:3)})\nonumber\\
		&\quad + \frac{1}{2}\sum_{\substack{I \sqcup J = [\ell]}} (-1)^{i(\gamma, I)} \langle \mathfrak{q}_{0,|I|}( \gamma^I) , \mathfrak{q}_{0,|J|}(\gamma^{J}) \rangle_L + (-1)^{|\gamma| + 1} \int_L i^*\mathfrak{q}_{\emptyset, l}(\gamma),
	\end{align*}
	where $i(\gamma,I) \coloneqq |\ga^I| + sign^\ga(I,[\ell]\setminus I)$.
\end{proposition}

\subsection{Bulk-deformations}\label{section: bulk-deformations and bounding cochains}
For a $\ZZ$-graded $\RR$-vector space $U$, let $\RR[[U]]$ be the ring of formal functions on the completion of $U$ at the origin. Explicitly, let $\{v_i\}_{i \in I}$ be a homogeneous basis for $U$, and $\{v_i^*\}_{i \in I}$ the dual basis for $U^*$. Letting $\{t_i\}_{i \in I}$ be formal variables of degree $-|v_i|$, there exists an isomorphism 
\begin{align*}
	\label{isomorphism coefficient ring}
	\RR[[\{t_i\}_{i \in I}]] &\xra{\cong} \RR[[U]],\notag \\
	t_i \;&\,\mapsto\;v_i^*
\end{align*}
of graded vector spaces. Each formal vector field $v \in \RR[[U]] \otimes U$ on $U$ can be viewed as a derivation $\partial_v\cl \CC[[U]] \rightarrow \CC[[U]]$. In coordinates, if $v = \sum_i f_i v_i$, for some $f_i \in \RR[[U]]$, then $\partial_v = \sum_i f_i \partial_{t_i}$.\par Define the vector fields 
\begin{equation*}
	\Gamma_U \coloneqq  \sum_i t_i \partial_{t_i}\qquad \text{ and } \qquad E_U \coloneqq  \sum_i \frac{deg(t_i)}{2}t_i\partial_{t_i}.
\end{equation*}
They are independent of the chosen basis. For $\ell \in \mathbb{Z}$, let $U[\ell]$ denote the graded vector space with $U[\ell]^i = U^{i+\ell}$. Let $S$ be another graded $\RR$-vector space and set
\begin{equation}\label{eq:coefficient-rings}
	Q \coloneqq  \Lambda_c[[U[2]]]\qquad \text{and} \qquad R \coloneqq  \Lambda_L[[U[2] \oplus S[1]]]
\end{equation}
Following \cite{ST16}, define the valuation $\nu\cl R \rightarrow \RR_{\geq 0}$ by\begin{equation}
	\label{valuation on coefficient ring}
	\nu \lbr{\sum_{j = 0}^{\infty} a_jQ^{\lambda_j} \prod_{i \in I} t_i^{\ell_{ij}} \prod_{k \in K} s_k^{p_{kj}}} = \inf_{\substack{j \\ a_j\neq 0}} (\lambda_j + \sum_{i \in I} \ell_{ij} + \sum_{k \in K} p_{kj}).
\end{equation} 
Let $\mathcal{I}Q = \{ f \in Q | \nu(f) > 0 \} \subset Q$ and define $\mathcal{I}R$ similarly.

\begin{definition}
	A \emph{bulk-deformation parameter} over $U$ is an element $\gamma \in \Omega^*(X;\mathcal{I}Q)$ with $d\gamma = 0$, $|\gamma| = 2$ and $[\gamma] = \Gamma_U \in Q \otimes U$. A \emph{bulk-deformation pair} over $U$ is a pair $(b,\ga)$, where $\gamma$ is a bulk-deformation parameter over $U$ and $b \in \Omega^*(L,\mathcal{I}R)$ has $|b| = 1$.
\end{definition}

\noindent
As in \cite[\textsection3.4]{ST16} bulk-deformed $\q$-operations are defined by
\begin{equation*}
	\mathfrak{q}^{b,\gamma}_{k,\ell}(\al_1, \dots, \al_k; \eta_1, \dots, \eta_\ell) \coloneqq  \sum_{\substack{s,t \geq 0\\s_0 + \dots + s_k = s}} \frac{1}{t!}\mathfrak{q}_{k+s,\ell+t}(b^{\otimes s_0}, \al_1 , \dots , \al_k , b^{\otimes s_k}; \eta_1 , \dots , \eta_\ell , \gamma^{\otimes t}).
\end{equation*}
for any bulk-deformation pair $(b,\gamma)$, while
\begin{equation*}
	\mathfrak{q}^{\gamma}_{\emptyset,\ell}(\eta_1, \dots, \eta_\ell) = \sum_t \frac{1}{t!}\mathfrak{q}_{\emptyset,\ell+t}(\eta_1 , \dots , \eta_\ell , \ga^{\otimes t})
\end{equation*}
and
\begin{equation*}
	\qb_{-1,\ell}(\eta_1, \dots, \eta_\ell) \coloneqq  \sum_{k,t} \frac{1}{t!(k+1)} \langle \q_{k,\ell+t}(b^{\otimes k};\otimes^\ell_{j = 1}\eta_j , \ga^{\otimes \ell}), b \rangle_L + \sum_{t} \frac{1}{t!}\q_{-1,t+\ell}(\otimes^\ell_{j = 1}\eta_j , \ga^{\otimes \ell}).
\end{equation*}

\begin{definition}
	Given any bulk-deformation pair $(b,\gamma)$, define the bulk-deformed $A_\infty$ algebra on $\Omega^*(L;\Lambda)$ by $\fm^{b,\gamma}_k \coloneqq  \fq^{b,\ga}_{k,0}$. 
\end{definition}

\begin{lemma}This is a  cyclic, unital curved $A_\infty$ algebra.\qed
\end{lemma}

\begin{definition}
	\label{def:weak-bounding-pair}
	A \emph{weak bounding pair} over $U$ is a bulk-deformation pair $(b,\ga)$ over $U$ such that \begin{equation}\label{mc-equation}
		\mb_{0} = c \cdot 1
	\end{equation}
	for some $c \in \mathcal{I}R$. We then say that $b$ is a \emph{bounding cochain} for $\ga$ and that $L$ is \emph{weakly unobstructed} over $U$. We call $c$ the \emph{Lagrangian superpotential}. The weak bounding pair $(b,\gamma)$ is \emph{point-like} if furthermore $S$ is one-dimensional and $\int_L b = s$, for the single coordinate $s$ on $S$.
\end{definition}

\begin{remark}
	To make sense of Definition \ref{def:weak-bounding-pair}, one first has to fix auxiliary data to define the $\q$ operations. However, it is shown in \cite{Fu10} that a pseudoisotopy of curved  $A_\infty$ algebra induces a quasi-equivalence of curved $A_\infty$ algebras, and hence an isomorphism on the space of bounding cochains. So the concept of $L$ being weakly unobstructed (over $U$) is independent of the choice of auxiliary data.
\end{remark}

\noindent
Following {\cite[Definition~2.23]{ST23}}, we introduce the following equivalence relation on weak bounding pairs.

\begin{definition}
	Two weak bounding pairs $(b_0,\gamma_0)$ for $\scA_0$ and $(b_1,\gamma_1)$ for $\scA_1$ are \emph{gauge equivalent} if there exists a pseudo-isotopy $\wt\scA$ from $\scA_0$ to $\wt\scA_1$ and a weak bounding pair $(\wt b,\wt\gamma)$ for $\wt \scA$ so that $j_i^*(\wt b,\wt\gamma) = (b_i,\gamma_i)$. We call $(\wt b,\wt\gamma)$ a \emph{pseudo-isotopy} from $(b_0,\gamma_0)$ to $(b_1,\gamma_1)$.
\end{definition}

\noindent
In general, the classification (or even existence) of bounding cochains is a difficult question. When $L$ has spherical cohomology, Solomon-Tukachinsky {\cite[Theorem~2]{ST21}} prove 
\begin{theorem}[]
	\label{thm:spherical cohomology implies unobstructed}
	Suppose $L$ satisfies $H^*(L;\RR) \cong H^*(S^n;\RR)$. Then the map 
	\begin{align}
		\rho: \{\text{weak bounding pairs}\}/\sim &\rightarrow \left(\hat{H}^*(X,L;\mathcal{I}Q)\right)_2 \oplus (\mathcal{I}R)_{n-1}\\
		\nonumber
		(b,\gamma) &\mapsto ([\gamma], \int_L b)
	\end{align}
	is a bijection. In particular, $L$ is weakly unobstructed over $W = \hat{H}(X,L;\RR)$.
\end{theorem}

 Their proof carries over without change, as our $\q$-operations satisfy the same properties.

\subsection{Open Gromov-Witten invariants}
\label{sec:OGW construction}
In the previous subsection, we considered the $\q$-operations for a single Lagrangian. Now we extend this definition to multiple Lagrangians, noting that we do not require the Lagrangians to intersect transversely.

\begin{definition}
	A \emph{generalised Lagrangian brane} is a tuple $$\LL = (\gamma, (L_1, \dots, L_m), b = (b_1, \dots, b_m)),$$ 
	where
	\begin{itemize}[leftmargin=20pt]
		\item $L_1, \dots, L_m$ are oriented Lagrangians, all of which are relatively spin with the same background class
		\item $\gamma$ is a bulk-deformation parameter over some vector space $U$;
		\item each $L_i$ is weakly unobstructed for the bulk deformation $\ga$ with bounding cochain $b_i$.
	\end{itemize}
\end{definition}

\noindent
Given a generalised Lagrangian brane $\LL$, consider the Novikov ring
\begin{equation*}
	\Lambda_\LL = \left\{  \sum_{i = 0}^{\infty} a_iQ^{\beta_i} | a_i \in \mathbb{R}, \; \beta_i \in H_2(X, \cup_i L_i), \; \lim_{i\to \infty } \omega(\beta_i) = \infty \right\}.
\end{equation*}
There are canonical inclusions $\Lambda_{L_i} \hookrightarrow \Lambda_\LL$. Let $R$ be the ring constructed in \eqref{eq:coefficient-rings}, using the Novikov ring $\Lambda_\LL$. Then, one can consider the combined $\q_{-1}$-operation
\begin{equation*}
	\q^{b,\ga}_{-1,\ell} \coloneqq  \sum_i \q^{b_i,\ga}_{-1,\ell}\cl \Omega^*(X;Q)^{\ell} \rightarrow R.
\end{equation*}
Here, $\q^{b_i,\ga}_{-1,\ell}$ denotes the $\q$-operation for the Lagrangian $L_i$, with bounding cochain $b_i$. One can check that the structure equations for the $\q_{-1}$ operation still holds. The construction of open Gromov-Witten invariants in \cite{ST21} only depends on the algebraic properties of the $\q$-operations, and can thus be adapted easily to our setup.\par 
We recapitulate the construction. Let $\LL$ be a generalised Lagrangian brane and define
\begin{equation}
	\begin{split}
		\mathfrak{i}\cl \Omega^*(X;Q) &\rightarrow R[-n]\\
		\eta \mapsto & (-1)^{n+ |\eta|} \int_{\LL} \eta
	\end{split}
\end{equation}
where $\int_\LL \eta \coloneqq  \sum_i \int_{L_i} \eta$. Let $C(\fii)$ denote the cone of $\fii$, equipping $R[-n]$ with the trivial differential. Furthermore, define $\hat{H}^*(X,\bL) = H^*(\ker(\mathfrak{i}))$.

\begin{definition}[{\cite[\textsection1.3.1]{ST23}}]
	\label{de:tensor-potential}
	Given the required auxiliary data and a weak bounding pair $(b,\gamma)$, the associated  \emph{relative potential} is the class 
	\begin{equation*}
		\Psi(\gamma,b) \coloneqq  [\q_{\emptyset,0}^\gamma, (-1)^{n} \q_{-1,0}^{\gamma,b}] \in H^*(C(\mathfrak{i})).
	\end{equation*}
	The \emph{tensor potential} $\nun =\cl H^*(C(\fii)) \rightarrow H^*(C(\fii))$ is the map induced on cohomology by 
	\begin{equation}\label{eq:tensor-potential}
		(\eta, a) \mapsto (\q_{\emptyset,1}^\ga(\eta), (-1)^{n+1} \qb_{-1,1}(\eta) - c \cdot a),
	\end{equation}
	where $c$ is given by Equation~\eqref{mc-equation}.
\end{definition}

\begin{remark}\label{} It follows from the properties of the $\fq$-operations that $\Psi(\gamma,b)$ and $\nun$ are well-defined. In the case of the relatively potential, this is \cite[Lemma~4.8]{ST23}, while the proof of \cite[Theorem~4]{ST23} shows that~\eqref{eq:tensor-potential} is a chain map.\end{remark}

\begin{definition}[{\cite[\textsection1.3.6]{ST23}}]\label{de:rel-quantum-conn}
	The \emph{relative quantum connection} is given by 
	\begin{equation*}
		\nabla_u(\Upsilon) \coloneqq  \partial_u \Upsilon + \partial_u\nun(\Upsilon)
	\end{equation*}
	for $u \in Q \otimes U$ and $\Upsilon \in H^*(C(\fii))$.
\end{definition}
	
Flatness of the relative quantum connection will follow from Lemma \ref{lem:tensor-potential-for-owdvv}.

\noindent
Fix a left inverse $$P_{\RR}\cl H^*(C(\fii_\RR)) \rightarrow\coker(\fii_\RR)$$ 
of the connecting map $\coker(\fii_\RR) \rightarrow  H^*(C(\fii_\RR))$ of the long exact sequence of the cone. Since $H^*(C(\fii_\RR)) \cong \hat{H}^*(X,L;\RR)$, this is equivalent to the map $\hat{H}(X,L;\RR) \rightarrow \RR$ in the introduction. As shown in \cite[\textsection4.4]{ST23}, this in turn induces a map 
\begin{equation}\label{}P\cl H^*(C(\fii)) \rightarrow\coker(\fii) \end{equation}
 The \emph{enhanced superpotential} is given by 
 \begin{equation}\label{} 
 	\cc \Omega(\gamma,b)\coloneqq  P\Psi(\gamma,b) \in\coker(\fii).
 \end{equation}

\noindent
This a priori depends on choices of auxiliary data, even though we omit them from the notation.

\begin{definition}[{\cite[\S1.3.3]{ST23}}]
	Suppose that $L$ admits a point-like weak bounding pair over $W \subset \hat{H}^*(X,L;\RR)$. Then define the \emph{open Gromov-Witten invariants} $\ogw$ to be the coefficients of the power series $\ov{\Omega}(\gamma,b)$ so that
	\begin{equation}
		\cc \Omega(\gamma,b) = \sum_{\substack{\beta\in H^2(X,L)\\k\ge 0\\r_i\ge 0}}\frac{T^\beta s^k \prod_{i}t_i^{r_i}}{k! \prod_{i}r_i!} \ogw_{\beta,k}(\otimes_{i}[\gamma_i]^{\otimes r_i}).
	\end{equation}
\end{definition} 

\noindent
Given any two choices $P_{\RR}, P_{\RR}'$, there exists a map $Q_\RR\cl H^*(X;\RR) \rightarrow \RR$ such that 
$$Q_\RR \circ \rho = P_\RR - P_\RR',$$ 
where $\rho\cl H^*(C(\fii_\RR))\rightarrow H^*(X;\RR)$ is the natural map.

\begin{proposition}[{\cite[Proposition~8.1]{Hollands_2025}}]
	For any $A_1, \dots, A_\ell \in \hat{H}^*(X,L)$, we have
	\begin{multline*}
		\normalfont\text{OGW}_{\beta,0}(A_1, \dots, A_\ell) - \normalfont\text{OGW}_{\beta,0}'(A_1, \dots, A_\ell)\\ = \sum_{\substack{\overline{\beta} \in \phi^{-1}(\beta)\\ i,j}} (-1)^{w_\mathfrak{s}(\beta)} g^{ij}\normalfont\text{GW}_{\overline \beta,\ell+1}(\Delta_i, \rho(A_1), \dots, \rho(A_\ell))\,Q_\RR(\Delta_j),
	\end{multline*}
	where $\phi\cl H_2(X) \rightarrow H_2(X,L)$ is the natural map, $(\Delta_i)_i$ is a basis for $H^*(X)$, and $g^{ij}$ is the inverse to the matrix given by $g_{ij} = \langle \Delta_i, \Delta_j \rangle_X$.
\end{proposition}

\noindent
Solomon-Tukachinsky prove certain properties for their open Gromov-Witten invariants, \cite[Proposition~4.19]{ST23} and \cite[Theorem~4]{ST21}. Since the $\fq$ operations defined in \S\ref{subsec:q-operations} satisfy the same relations as their operations, the proof caries over verbatim.

\begin{proposition}\label{prop:OGW-axioms}
 The invariants $\ogw$ satisfy the following axioms for $A_1,\dots,A_\ell\in W$.
	\begin{enumerate}[\normalfont 1),leftmargin= 20pt,ref=\arabic*]
		\item (Symmetry) For any permutation $\sigma \in S_\ell$ we have
		$$\ogw_{\beta,k}(A_1,\ldots,A_\ell)=(-1)^{s_\sigma(A)}\ogw_{\beta,k}(A_{\sigma(1)},\ldots,A_{\sigma(\ell)}),$$
		where $s_\sigma(A)\coloneqq  \sum_{\substack{i>j\\ \sigma(i)>\sigma(j)}} |A_{\sigma(i)})|\;|A_\sigma(j))|$.
		\item (Degree) We have
		$\ogw_{\beta,k}(A_1,\ldots,A_\ell)=0$ unless
		\begin{equation*}\label{ax_deg}
			n-3+\mu(\beta)+k+2\ell = kn+ \sum_{j=1}^\ell |A_j|.
		\end{equation*}
		\item (Fundamental class)
		\begin{equation*}\label{ax_unit}
			\ogw_{\beta,k}(1,A_{1},\ldots,A_{\ell-1})=
			\begin{cases}
				-1, & (\beta,k,\ell)=(\beta_0,1,1),\\
				P_\RR(A_1), & (\beta,k,\ell) = (\beta_0,0,2),\\
				0, & \textit{otherwise}.
			\end{cases}
		\end{equation*}
		\item (Zero) If $\beta= \beta_0$, then
		\begin{equation*}\label{ax_zero}
			\ogw_{\beta_0,k}(A_1,\ldots,A_\ell)=
			\begin{cases}
				-1, & (k,\ell)=(1,1)\text{ and } A_1=1,\\
				P_\RR(A_1 \smallsmile A_2), & (k,\ell) = (0,2),\\
				0, & \textit{otherwise}.
			\end{cases}
		\end{equation*}
		\item (Divisor)
		If $|A_\ell| = 2,$ then
		\begin{equation*}\label{ax_divisor}
			\ogw_{\beta,k}(A_1,\ldots,A_{l})=\int_\beta A_\ell \cdot\ogw_{\beta,k}(A_1,\ldots,A_{\ell-1}).
		\end{equation*}
	\end{enumerate}
\end{proposition}

	\subsection{Open WDVV equations}\label{subsec:owdvv}
The WDVV equations are recursion relations derived from the Splitting Axiom of classical GW theory. Suppose now $\bL$ is a generalised Lagrangian brane as in the previous subsections. Its open Gromov--Witten invariants satisfy an analogous recursion relation called the open WDVV equation. This was first proven in \cite{ST23} for the invariants of \cite{ST21}. In contrast to the properties of Proposition~\ref{prop:OGW-axioms}, the open WDVV equation does not follow directly from the results in \textsection\ref{sec:props-of-q-ops}, but, in our generality, requires \cite[\textsection 3]{HH26}, recalled below.

\medskip

\noindent
As in \cite{ST23b}, we need the following data to state the open WDVV equations
\begin{enumerate}
	\item a subspace $U \sub H^*(X;\bR)$ such that $U\otimes Q_U$ is a Frobenius subalgebra of $\text{QH}_U(X)$,
	\item a bounding pair $(\gamma_W,b)$ over $W = \rho\inv(U) \sub \wh{H}^*(X,L;\bR)$, where $b$ is point-like and $\rho \cl \wh H^*(X,L;\bR)\to H^*(X;\bR)$ is the natural map.
\end{enumerate}
Given this, we can define $W' \coloneqq \ker(P_\bR|_W) \sub W$ to be the image of the connecting homomorphism $y\cl \RR \rightarrow \hat{H}^*(X,L)$ of the long exact sequence induced by $\int_L\cl \Omega^*(X)\to \bR$. 
The restriction $\rho|_{W'}$ is injective, so we may choose index sets $I_{W'}\sub I_U$, a bases $\{\Delta_i\}_{i\in I_U}$ of $U$ and $\{\Gamma_i\}_{i\in I W'}$ of $W'$ respectively such that $\rho(\Gamma_i) = \Delta_i$ for $i \in I_{W'}$. Denote by 
\begin{equation*}\label{}
	\del_i\cl Q_U \to Q_U\qquad \text{and}\qquad\del_i \cl R_W\to R_W
\end{equation*}
the derivations corresponding to $\Delta_i$ with $i \in I_U$, respectively $\Gamma_i$ with $i \in I_{W'}$. 
Let
$$g_{ij}\coloneqq \int_X\Delta_i \cdot\Delta_j$$ 
for $i,j\in I_U$ and define $(g^{ij})_{i,j}\coloneqq (g_{ij})\inv$ to be the inverse matrix, which exists by assumption (1). Abbreviate
$\Phi \coloneqq \Phi_U \in Q_U$ and $\ov\Omega\coloneqq \ov\Omega(\gamma_W,b)\in R_W$. Let $\rho^*\cl Q_U \to Q_W$ be the induced ring homomorphism. 

\begin{theorem}[Open WDVV equations]\label{thm:owdvv-equation}
	Let $c$ be the coefficient of the Maurer-Cartan equation~\eqref{mc-equation} for the bounding pair $(\gamma_W,b)$ and fix $u,v\in W\oplus S$ and $w \in W$. Denote by $u_W$ and $v_W$ their projections to $W$ and write $\ov{w} = \rho(w)$, $\ov{u} = \rho(u_W)$ and $\ov{v} = \rho(vW)$. Then, the enhanced superpotential $\ov\Omega = P(\Psi(\gamma,b))$ and $\Phi$ satisfy
	\begin{equation}\label{eq:owdvv} 
		\del_u c\cdot\del_w\del_v\ov\Omega \,-\,\s{\substack{i\in I_{W'}\\j\in I_U}}{\del_u\del_i\ov\Omega \cdot g^{ij}\cdot \rho^*\del_j\del_{\ov{v}}\del_{\ov w}\Phi} 
		\;=\;\del_u\del_w\ov\Omega\cdot\del_v c\, -\, \s{\substack{i\in I_U\\j\in I_{W'}}}{\rho^*\del_{\ov u}\del_{\ov w}\del_i\Phi \cdot g^{ij}\cdot \del_j\del_{{v}}\ov\Omega}. 
	\end{equation}
\end{theorem}

To see the reason for its name, recall that the WDVV equations in usual GW theory can be expressed as saying that the GW potential $\Phi$ satisfies the equation
	\begin{equation}\label{eq:wdvv} 
		\s{i,j}{\del_u\del_v\del_i\Phi\cdot g^{ij}\cdot\del_j\del_{w}\del_{z}\Phi} \;=\; 	\s{i,j}{\del_u\del_w\del_i\Phi\cdot g^{ij}\cdot\del_j\del_{v}\del_{z}\Phi}
	\end{equation}

%
%

	As in the closed case, the OWDVV equation is due to the boundary structure of certain moduli spaces with a certain geodesic constraint, cf \cite[\S3]{ST23}.Concretely, these `geodesic operations', denoted $\fq_{k, \ell; \chi}$, are associated to moduli spaces of stable discs, where three marked points are required to lie on a geodesic. We will outline the case with two interior and one boundary marked point; the other case is similar. Consider the moduli space \[
\Mbar_{k+1,\ell; 1,2}(\beta) := [0,1] \times_{\Mbar_{1,2}}\Mbar_{k+1, \ell}(\beta),
\]
where the map $\Mbar_{k+1,\ell}(\beta) \rightarrow \Mbar_{1,2}$ is the stabilisation map, which forgets the map to $X$ and the required number of marked points. The image of the analytic map $\sigma\cl [0,1] \rightarrow \Mbar_{1,2}$ is given by the locus of points where the boundary marked point and the two interior marked points lie on a geodesic in a specified order. Then, $\fq^{\beta}_{k, \ell; 1,2}$ is defined by the pullback and pushforward of differential forms as in \ref{subsec:q-operations}, using the correspondence given by $\Mbar_{k+1,\ell;1,2}(\beta)$.

\begin{remark}
	The fibre product $\Mbar_{k+1,\ell}(\beta)$ need not be a smooth orbifold with corners, as the stabilisation map is not a submersion. This means that one cannot directly use them to define a pushforward of differential forms. In all the cases where Solomon--Tukachinsky apply the open WDVV equation, the spaces $\Mbar_{k+1,\ell}(\beta)$ are actually real algebraic, and it is thus possible to apply a real algebraic/analytic resolution of singularities to $\Mbar_{k+1,\ell}(\beta)$ to define the operations. This is outlined more precisely below.
\end{remark}
Using \cite{HH26}, we associate an operation (on differential forms) to any analytic chain in $\Mbar_{k,\ell}$, using the resolution of singularities (\S A op cit) mentioned above. In order to achieve strict compatibility of these operations with respect to the natural clutching maps on the moduli spaces $\Mbar_{k,\ell}$, \cite{HH26} constructs a specific chain model on the space $\Mbar_{k,\ell}$. This already appeared in \S\ref{subsec:q-operations} when we defined the $\fq$-operations. We recall the set-up briefly.

\begin{remark}\label{} 
	 In contrast to \cite{HH26}, we will phrase the compatibility here in terms of algebras over (coloured) operads. One can directly see from the construction that the compatibility conditions agree. See \cite[\S 1.4 and Lemma~2.40]{HH26} for a similar discussion in the case of cohomological field theories and $A_\infty$ algebras.
\end{remark}

By a \emph{stable tree} we mean the combinatorial data associated to a nodal genus $0$ Riemann surface with marked points and at most one boundary component. To each such stable tree $\Gamma$ we associate the product moduli space
\begin{equation}
	\Mbar_{\Gamma} =  \prod_{v \in V(\Gamma)} \Mbar_{\mfa_v}.
\end{equation}
It admits a canonical map to $\Mbar_{\mathsf{a}_\Gamma}$, which descends to an isomorphism 
\[
\del_\Gamma\Mbar_{\mathsf{a}_\Gamma} \cong \modulo{\Mbar_{\Gamma}}{\Aut(\Gamma)}.
\]
We abbreviate
$$\cMc_\Gamma\coloneqq I^{E(\Gamma)}\times\Mbar_\Gamma.$$
When $\Gamma$ only has boundary edges, then $\cMc_\Gamma$ obtains a canonical orientation via the open embedding $\cMc_\Gamma \to \Mbar_{\mathsf{a}}$ given by a collar neighbourhood of the boundary stratum $\del_\Gamma\Mbar_{\mathsf{a}}$. These moduli spaces together with the clutching maps form a coloured operad with object set $\bN^2$ in the category of smooth manifolds with corners. In \cite[\textsection 3.1]{HH26}, we construct a special chain model for the singular homology of these moduli spaces that allows us to obtain strict operations. Similar to how we used the cubical cobordisms in the definition of the $\fq$-operations, the idea is to `thicken' the usual singular chains on moduli spaces of stable curves by incorporating (real) analytic chains on the moduli spaces $\Mbar_\Gamma$. Observe that a clutching map induces a well-defined surjection 
\begin{equation}\label{eq:induced-by-clutching} 
	\cMc_{\Gamma}|_{\{t_E = 0\}}\to \cMc_{\Gamma'}|_{\{t_{E'} = 0\}}
	\end{equation}  
exactly if there exists a contraction $\Gamma\to \Gamma'$ that contracts exactly the edges in $E\sm E'$. Since these maps are real analytic, the chain complex
$$\dmc_*(\mfa) \coloneqq  \varinjlim\limits_{(\Gamma,E)}C^{an}_*(\cMc_\Gamma|_{\{t_E = 0\}}),$$
\noindent where the structure maps are given by the pushforwards $$C^{an}_*(\cMc_\Gamma|_{\{t_E = 0\}})\to C^{an}_*(\cMc_{\Gamma'}|_{\{t_{E'} = 0\}})$$ along the maps~\eqref{eq:induced-by-clutching}, is well-defined

\begin{definition}[{\cite[Definition~11.2.1]{Yau16}}]\label{de:coloured-operads}
	Recall that an \emph{$\fC$-coloured operad} $\{\cO(c_1,\dots,c_n;c)\}_{c_i,c\in \fC}$ (in a symmetric monoidal category $(\cC,\otimes)$) consists of a set $\fC$ of colours and for each $(n+1)$-tuple $(c_1,\dots,c_n;c)$ of colours an object $\cO(c_1,\dots,c_n;c)$ of $\cC$ together with maps 
	\begin{equation}\label{eq:operad-composition} 
		\cO(d_1,\dots,d_n,d)\otimes \cO(c^{(1)}_1,\dots,c^{(1)}_{k_1};d_1)\otimes \dots \otimes \cO(c^{(n)}_1,\dots,c^{(n)}_{k_n};d_n)\to \cO(c^{(1)}_1,\dots,c^{(n)}_{k_n};d)
	\end{equation}
	satisfying analogues of the operad axioms. A coloured operad is \emph{symmetric} if there exists for each $c_1,\dots,c_n,c\in \fC$ and any permutation $\sigma\in S_n$ an isomorphism $$\sigma^*\cl\cO(c_1,\dots,c_n;c)\to \cO(c_{\sigma(1)},\dots,c_{\sigma(n)};c)$$
	and these isomorphisms are compatible with the composition maps~\eqref{eq:operad-composition}. It is \emph{unital} if for each $c \in \fC$ there exists a map $1_\cC\to \cO(c;c)$ satisfying the unity axioms on \cite[p.178]{Yau16}.
\end{definition}

\noindent
Refer to \cite[Chapter~11]{Yau16} for more details on coloured operads.

\begin{lemma} The collection $\set{\dmc_*(\mathsf{a})}_{\mathsf{a}\in \ov\cA}$ defines a unital symmetric coloured operad $\dmc_*$ in the category of chain complexes.\qed 
\end{lemma}

\begin{proof} Take the set of colours to be $\fC\coloneqq\{i,\bdr \}$, corresponding to the division of marked points into boundary and interior marked points. Then, we define 
	\begin{equation}\label{}\cO(\bdr^{\otimes k},i^{\otimes \ell};x) \coloneqq
		\begin{cases}
			\dmc_*(k,\ell) & \quad \text{if }x = b\\
			\dmc_*(\emst,\ell) & \quad \text{if } x = i, k = 0\\
			0 &\quad \text{otherwise}
	\end{cases}\end{equation} 
	and extend it in the obvious way to arbitrary finite words in $\partial$ and $i$. The composition maps are induced by clutching maps between moduli spaces of (open) stable curves, whence the compatibility relations are tautologically true. Unitality follows from the fact that we use the coarse moduli spaces of spheres/discs with one incoming and one outgoing (boundary) marked point.
\end{proof}

\begin{definition}[{\cite[Definition~13.2.3]{Yau16}}]\label{de:coloured-operad-algebra} 
	An \emph{algebra over a coloured operad} $\cO$ in $(\cC,\otimes)$ with set of colours $\fC$ is a set $X = \{X_c\mid c \in \fC\}$ of objects of $\cC$ together with the data of maps $$\cO(c_1,\dots,c_n;c)\otimes \bigotimes\limits_{i =1}^n X_{c_i}\to X_c$$ that are compatible with the compositions maps of $\cO$.
\end{definition}

\noindent
Given an admissible simplex $\sigma \cl \Delta^m \to \cMc_{\ov\Gamma}$ and a stable map graph $\Gamma$ with stabilisation $\Gamma^{\stb} =\ov\Gamma$, we define
\[
\cK_{\Gamma,\sigma} \coloneqq\cKc_{\Gamma} \times_{\cMc_{\ov\Gamma}}\Delta^m = (\cTc_{\Gamma} \times_{\cMc_\Gamma}\Delta^m,p_\sigma^*\cEc_\Gamma,p_\sigma^*\obs_\Gamma),
\]
\noindent where $p_\sigma \cl \cT_{\Gamma,\sigma} := \cTc_{\Gamma} \times_{\cMc_\Gamma}\Delta^m\to \cTc_\Gamma$ is the canonical projection and  the map 
\begin{equation}
	\label{eq:chain-fibre-product}
	\cTc_{\Gamma}\to \cMc_{\ov{\Gamma}}
\end{equation}
is given by the stabilisation map $\cTc_{\Gamma}\to \Mbar_{\ov{\Gamma}}$ and the composition 
$\cTc_{\Gamma}\to I^{E(\Gamma)} \rightarrow I^{E(\cc{\Gamma})}.$\par
Since the map~\eqref{eq:chain-fibre-product} is not a submersion, $\cK_\sigma$ is not a global Kuranishi chart in general. However, the singular points of~\eqref{eq:chain-fibre-product} lie in strata of codimension at least $1$. In \cite[\textsection A]{HH26}, we replace $\cK_\sigma$ by a derived orbifold constructed by resolving the singularities of \eqref{eq:chain-fibre-product}. The projection $p_\sigma$ induces a rel--$C^\infty$ morphism of derived orbifolds to $\cKc_\Gamma$. For ease of notation, we denote the resolved orbifold also by $\cK_{\Gamma,\sigma}$ and the map to $\cKc_\Gamma$ by $p_\sigma$. 

\begin{definition}\label{de:q-operations-intro} The $\fq$-operation associated to an admissible simplex $\sigma\cl \Delta^m \to \cMc_{\cc{\Gamma}}$ is 
	\begin{equation*}\label{eq:q-operation}
		\fq_\sigma\cl \Omega^{*}(L;\Lambda)^{\otimes k_{\ov\Gamma}}\otimes \Omega^{*}(X;\Lambda)^{\otimes \ell_{\ov\Gamma}}\to \Omega^{*}(Y_\sigma;\Lambda)
	\end{equation*}
	defined by
	\begin{equation*}
		\label{} \fq_\sigma(\alpha) \coloneqq \sum_{\substack{\Gamma \in \stb^{-1}(\cc{\Gamma}) }}
		{(\eva^\Gamma_{0,\sigma}\g p_\sigma)_*\lbr{p_\sigma^*(\eva^\Gamma)^*\alpha \wedge p_\sigma^*\obs^*\eta_{\Gamma}}}\,Q^{\beta(\Gamma)},
	\end{equation*}
	where $Y_\sigma = X$ if $\ov\Gamma$ encodes a closed curve and $Y_\sigma = L$ if $\ov\Gamma$ encodes an open curve.
\end{definition}

\noindent
We extend the definition of the operations $\fq_\sigma$ linearly to admissible chains. By \cite[Lemma~3.23]{HH26}, these operations descend to $\dmc_*$. However, they will generally not satisfy $d\fq_{\sigma}-\fq_\sigma d = \fq_{\del\sigma}$, just as the $A_\infty$ algebra defined in the previous section might be curved. Thus, we deform these operations by a bounding cochain. 
For this, we have to define the lift of a chain $\sigma \cl \Delta^k \to \cMc_{\Gamma}$. Let $\Gamma^+$ be a stable graph which contracts to $\Gamma$ after forgetting a single incoming boundary marked point. Then, by taking an analytic triangulation of the resolution of the fibre product $\Delta^k\times_{\cMc_{\Gamma}}\cMc_{\Gamma^+}$, one obtains a chain $\pi^*_{\Gamma^+} \sigma$. The same construction works when $\Gamma_+$ contracts to $\Gamma$ after forgetting multiple interior and boundary marked points.

\begin{definition}\label{} The \emph{deformed $\fq$-operation} associated to $\sigma \in \dmc_n$ is 
	\begin{equation}\label{} 
		\fq^{b,\gamma}_{\sigma}(\alpha) = \s{\Gamma_+}{\frac{1}{\ell!}\fq_{\pi_{\Gamma_+}^*\sigma}(b^{i_0},\alpha_1,b^{i_1},\dots,\alpha_k,b^{i_k};\gamma^{\otimes \ell})}
	\end{equation}
	while 
	\begin{equation}\label{} 
		\fq^{b,\gamma}_{1,0}(\alpha) = \fq_{1,0}(\alpha) + \s{\Gamma_+}{\frac{1}{\ell!}\fq_{\pi_{\Gamma_+}^*D_1}(b^{i_0},\alpha,b^{i_1}, \gamma^{\otimes \ell})}
	\end{equation}
	where $\Gamma_+$ contracts to $\sigma$ (or to the disc with one incoming and one outgoing marked point), after forgetting the marked points which take in $b$ or $\gamma$.
\end{definition}

\begin{theorem}
	\label{thm:closed-open-algebra}
	Given a bounding pair $(b,\gamma)$, the operations $\fq^{b,\gamma}$ of Definition~\ref{de:q-operations-intro} form the structure maps of an algebra $\scA$ over $\dmc_*$ with 
	$$\scA(o) = (\Omega^*(L;\Lambda),\fq_{1,0}^\fb)\qquad \text{and}\qquad\scA(i) = (\Omega^*(X;\Lambda),d).$$ 
	Moreover, the action of $\dmc_*$ is compatible with the forgetful maps.
\end{theorem}

\begin{proof}
	When $\gamma = 0$, this follows by observing that the operations $\fq^{b,0}$ are the restrictions of the open-closed DMFT $\scF^b_L$ of \cite[Corollary~3.2]{HH26}. Since $\scF^b_L$ is operadic, its restriction to $\dmc_n$ (as defined here) is a symmetric monoidal dg functor, which exactly yields the statement. The same statement holds after additionally bulk deforming by the insertions $\gamma$.
\end{proof}

\begin{proof}[Proof of Theorem~\ref{thm:owdvv-equation}]
We now discuss how this can be used in conjunction with \cite[\S 5.3]{ST23} to prove Theorem~\ref{thm:owdvv-equation}. By computation in the proof of \cite[Theorem~3]{ST23}, the equality in~\eqref{eq:owdvv} reduces to showing the following property of the map~\eqref{eq:tensor-potential}, which we also denote by $\nun$.

\begin{lemma}\label{lem:tensor-potential-for-owdvv} Given formal vector fields $u,v \in Q_W\otimes R\oplus R_W\otimes S$, the map $\del_u\nun\g \del_v\nun$ is chain homotopic to $(-1)^{|u||v|}\del_v\nun\g \del_u\nun$.
\end{lemma}

\noindent
This is \cite[Theorem~5]{ST23} and the homotopy is given (on p.45 op cit) by
\begin{align*}
	H_{uv}(\eta,\xi)=
	\lbr{(-1)^{1+\star} \fq^{\gamma}_{\emptyset,3;\chi_0}(\del_u\gamma,\eta,\del_v\gamma),
	&(-1)^{\star} \fq^{\gamma,b}_{-1,3;0,3}(\del_u\gamma,\eta,\del_v\gamma)\\
	&\;+(-1)^{(|\eta|+1)(|u|+|v|+1)} \langle\fq^{\gamma,b}_{0,2;1_0,2}(\eta,\del_u\gamma),\del_vb\rangle\\
	&\;+(-1)^{|\eta|+(|\eta|+1)(|u|+|v|)+|u||v|} \langle\fq^{\gamma,b}_{0,2;1_0,2}(\eta,\del_v\gamma),\del_ub\rangle\\
	&\;+(-1)^{|\eta|+(|\eta|+1)(|u|+|v|)+|u||v|} \langle\fq^{\gamma,b}_{1,1;2_{0,1},1}(\del_vb;\eta),\del_ub\rangle},
\end{align*}
where $\star = |\eta|+|u|+|v|+|\eta||v|$ and the $\fq$-operations are defined by bulk-deforming with $\gamma$ via the operations associated to certain chains. For this, \cite{ST23} observe that given any $k\ge -1$ and $\ell\ge 0$ in the stable range and any three marked points $w_1,w_2,w_3$ with $w_2$ interior there exists a cross ratio map 
\begin{equation}\label{} 
	\chi_{w}\cl \Mbar_{k,\ell}\to \bP^1 = \bC \cup \{\infty\} : \frac{(w_3-\cc{w_2})(w_2-w_1)}{(w_3-w_1)(w_2-\cc{w_2})}
\end{equation} 
Its existence follows from a similar arguments in \cite[Appendix~D]{MS12}, where it is shown how to extend cross ratios from $\cM_{\emst,\ell+1}$ to the compactification.
\begin{itemize}[leftmargin=20pt]
	\item To define $\fq_{-1,3+\ell;0,3}$, with three selected interior marked points (the others take as input $\gamma$) we take $w_1,w_2,w_3$ to be these selected interior marked points to be constrained to lie on a geodesic; that is, consider the fibre product $[0,1]\times_{\bP^1}\Mbar_{-1;3+\ell}$.
	\item For $\fq_{0,2+\ell;1_0,2}$, we take $w_1$ to be the outgoing boundary marked point and let $w_2$ and $w_3$ be the selected interior marked points, the others taking as input $\gamma$.
	\item For $\fq_{1,1+\ell;2_{0,1},1}$, we take $w_1$ to be the outgoing boundary marked point, $w_3$ to be the unique incoming boundary marked point and let $w_2$ be a fixed selected interior marked point. 
	\item To define $\fq_{\emptyset,3+\ell;\chi_0}$, let $\Mbar_{\emst,4;\chi_0}$ be the preimage of $[0,1]\sub \bP^1$ under the usual cross ratio of the marked points, and let $\Mbar_{\emst,4+\ell;\chi_0}$ be its preimage under the map forgetting the additional $\ell$ marked points.
\end{itemize}
In either case, given a stable tree $\Gamma$ contracting onto $\Mbar_{i,j}$ in this list, we define the fibre products $\cMc_{\Gamma,\chi} = [0,1]\times_{\bP^1}\cMc_\Gamma$ and similarly for $\cMc_{\Gamma,\chi_0}$. 

 \begin{remark}\label{rem:right-trees} 
 	There is a slight subtlety here in the choice of graphs we allow. In \cite{HH25} and \cite{HH26}, the half-edges of the graphs $\Gamma$ are equipped with an orientation sign, indicating whether it is associated to an incoming or outgoing marked point. If we allowed for all possible combinations of these orientation signs here, it would lead to over counting of the operations. To avoid that, we declare that in the case with one outgoing marked point, be it on the boundary or interior, the orientation signs of the half-edges have to point towards that `root'; in particular, there is no edge whose half-edges have the same orientation sign. In the case of $\fq_{-1,3+\ell}$, comparing ordering of the interior marked points on the geodesic with their labelling indicates which is the `starting' point of the geodesic. We then only consider the graphs $\Gamma$ where all half-edges are oriented towards the disc carrying that starting point.
 \end{remark}
 
 For each tree as in Remark~\ref{rem:right-trees}, we can resolve the fibre product using \cite[\S A]{HH26}. Choosing any real analytic triangulation, we obtain a chain $\sigma_{\Gamma,\chi}\in C^{an}_*(\cMc_\Gamma)$. By the proof of \cite[Lemma~2.14]{HH26}, this chain is independent up to degenerate chains of the choice of resolution and the choice of triangulation. Since push-pull operations defined using degenerate chains vanish by \cite[Proposition~3.14]{HH26}, we can define 
 \begin{equation}\label{} 
 	\fq_{-1,3+\ell;0,3} \,\coloneqq\, \s{\beta,\Gamma}{\fq^\beta_{\sigma_{\Gamma,\chi}}Q^\beta}
 \end{equation}
 where we are summing over all stable trees as in Remark~\ref{rem:right-trees} labelling the boundary strata of $\Mbar_{-1,3+\ell}$. The other operations appearing in the definition of $H_{uv}$ are defined analogously. 
Theorem \ref{thm:closed-open-algebra} then implies that these operations satisfy the same relations as those in \cite[\S3.2]{ST23}, that is, the boundary structure of the respective moduli spaces (shown in Figure~8 op cit) translates into the desired algebraic relation. In particular, $H_{uv}$ is indeed a homotopy between $\del_u\nun\g \del_v \nun$ and $(-1)^{|u||v|}\del_u\nun\g \del_v \nun$. The remaining arguments in \cite[\S5.3]{ST23} now carry over verbatim and yield the proof.
\end{proof}

	\section{Lagrangian cobordisms}\label{sec:lag-cob}

For notational convenience define $\bC_1 := \{z\in \bC\mid \Re(z) \in [-1,1]\}$.

\begin{definition}
	\label{def:Lag cob} Suppose $(X,\omega)$ is a closed symplectic manifold and $L_1^\pm,\dots,L_{m_\pm}^\pm \sub X$ are embedded Lagrangians.  A \emph{Lagrangian cobordism} $\wt L$ from $(L_1^-,\dots,L_{m_-}^-)$ to $(L_1^+,\dots,L_{m_+}^+)$ is a Lagrangian submanifold of $(X\times \bC_1,\omega\oplus \omega_{\text{std}})$ so that
	\begin{equation}\label{eq:negative-end}
		\wt L \cap (X\times [-1,-1+\epsilon)\times i\bR = (L_1^- \sqcup \dots \sqcup L^-_{m_-}) \times [-1,-1+\epsilon)
	\end{equation}
	and
	\begin{equation}\label{eq:positive-end}\wt L \cap (X\times (1-\epsilon,1]\times i\bR) = (L^+_1 \sqcup \dots \sqcup L^+_{m_+}) \times (1-\epsilon,1]\end{equation}
	for some $0 < \epsilon < \frac13$. We write $\wt L \cl L_*^- \leadsto L^+_*$ and say they are \emph{Lagrangian cobordant}.
	
	A \emph{Lagrangian pseudoisotopy} is a Lagrangian cobordism $\wt L$ from a Lagrangian $L$ to itself which is diffeomorphic to $L \times [-1,1]$.
\end{definition}

Even though there is no interesting change in topology, Lagrangian pseudoisotopies are already very useful for proving invariance statements; see Proposition~\ref{thm:trivialisable cobordism induces quasi-equivalence}. Hamiltonian isotopies induce Lagrangian pseudoisotopies.

\begin{example}\label{ex:suspension-hamiltonian-isotopy} Given a Lagrangian $L\sub (X,\omega)$ and a Hamiltonian $H_t$ on $X$ with $H_t \equiv H_{\pm 1}$ on $[-1,-1+\epsilon]$, respectively $[1-\epsilon,1]$, let $\phi^t_H$ be the associated Hamiltonian isotopy. Then the image $L_H$ of the embedding
	\begin{equation}\label{}L \times [-1,1]\hkra X\times\bC : (x,t)\mapsto \lbr{\phi_H^t(x),t + \text{i}H_t(\phi_H^t(x))}\end{equation}
	is a Lagrangian isotopy from $\phi_H^{-1}(L)$ to $\phi^1_H(L)$, called the \emph{suspension of the Hamiltonian isotopy}.
\end{example}

\subsection{Global Kuranishi charts in the case of Lagrangian cobordisms}\label{subsec:gkc-lag-cobordism}
Let $\wt J$ be an $\wt\omega$ tame almost complex structure on $X\times\bC_1$ which is of the form $J_\pm \times\text{i}$ over $X\times [-1,-1+\epsilon)\times\bR$, respectively $X\times (1-\epsilon,1]\times\bR$ for some $J_\pm\in \cJ_\tau(X,\omega)$. Let $\wt X = X \times \CC_1$. We call $\wt J$ \emph{cylindrical} if it is of the form 
$$\wt{J}_{(x,s+\text{i}t)} = J_s \times\text{i}\qquad \quad (x,s+\text{i}t)\in \wt X.$$ 
By \cite[Lemma~4.2.1]{BC13}, the image of any compact $\wt J$-holomorphic curve $u \cl (C,\del C)\to (\wt X,\wt L)$ under the projection map $\pr_\bC\cl X\times\bC_1\to \bC_1$ is contained in a compact subset and constant if it meets $([-1,-1+\epsilon)\times \bR \cup [-1,-1+\epsilon)\times \bR$.

Fix now $\epsilon > 0$ so that \eqref{eq:negative-end} and \eqref{eq:positive-end} hold for $2\epsilon$ and let $j_\epsilon \cl X\hkra \wt X : x\mapsto (x,\epsilon)$ be the inclusion. Let $\wt\beta \in H_2(\wt X,\wt L;\bZ)$ be arbitrary and fix $g,h,k\geq 0$ and $\ell_1,\dots,\ell_h \geq 0$. We equip the moduli space $\Mbar_{g,h;k,\ell}^J(\wt X,\wt L,\wt\beta)$ with the Gromov topology.\par 
Given $L_i^\pm$, define 
$$\fB^\pm_i := \{\beta \in H_2(X,L_i^\pm;\bZ) \mid {j_\epsilon}_*\beta = \wt\beta\}$$
and define 
$$\Mbar_{k,\ell}(\wt\beta;L^\pm_*) := \djun{i = 1}^{m_\pm} \djun{\beta \in \fB^\pm_i}\Mbar_{k,\ell}(\beta,L^\pm_i) .$$

\begin{lemma}\label{lem:cobordism-moduli-compact} The moduli space $\Mbar_{k,\ell}(\wt\beta) $ of stable spheres or discs with boundary on $\wt L$ is compact.
\end{lemma}

\begin{proof} This follows from \cite[Lemma~4.2.1]{BC09b} and Gromov compactness for compact curves with Lagrangian boundary condition, \cite{Ye94}.
\end{proof}

We need one more definition to properly describe the relative smoothness of the global Kuranishi chart we construct for $\Mbar_{k,\ell}(\wt\beta)$.

\begin{definition}\label{} A \emph{rel--$C^\infty$ manifold with boundary} $Y/T$ is a continuous map $p \cl Y\to T$ of topological spaces, so that $Y$ admits an open cover $\{U_i\}\iI$, where for each $i \in I$, the image $p(U_i)$ is open in $T$ and there exists a homeomorphism $\varphi_i$ fitting into the commutative diagram
	\begin{center}\begin{tikzcd}
			U_i \arrow[rr,"\varphi_i"] \arrow[dr,"p"]&&p(U_i)\times \Omega\arrow[dl,"\pr_1"]\\ & p(U_i) \end{tikzcd} \end{center}
	for some open subset $\Omega \sub \{x \in \bR^n\mid x_n \geq 0\}$, so that $$\varphi_i\g(\varphi_j|_{\varphi_j(U_i\cap U_j)})\inv\cl\varphi_j(U_i\cap U_j)\to \varphi_i(U_i\cap U_j)$$ 
	extends to a rel--$C^\infty$ diffeomorphism between neighbourhoods of domain and target.
\end{definition}

\begin{remark}
	In the above definition, the fibres of $Y$ over $T$ are allowed to be manifolds with boundary. In \cite[Definition~C.8]{HH25}, the rel--$C^\infty$ manifolds are such that the base $T$ is allowed to have corners, so that, informally, `all the boundary comes from the base'.
\end{remark}

\begin{lemma}\label{thm:cobordism-global-chart-existence} Let $k,\ell\geq 0$, or $k = \emst$, and, $\wt \beta \in H_2(\wt X,\wt L;\bZ)$, respectively $\wt\beta \in H_2(\wt X;\bZ)$.
	\begin{enumerate}[\normalfont 1),leftmargin=20pt,ref=\arabic*]
		\item $\Mbar_{k,\ell}(\wt\beta)$ admits an oriented smooth global Kuranishi chart 
		$$\wt \cK=  (\wt G,\wt \cT/\wt \cB,\wt \cE,\wt \obs)$$ with corners that has the expected virtual dimension.
		\item The boundary evaluation maps $\evab_j\cl \wt\cT \to  \wt L$ are submersions.
		\item\label{i:trivialisable} If $\wt L = L \times [-1,1]$ for a possibly disconnected Lagrangian $L \sub X$, then $\wt \cT$ admits a $G$-invariant submersion to $[-1,1]$ so that the fibre over each point is a rel--$C^\infty$ global Kuranishi chart for $\Mbar_{k,\ell}^{\,J_s}(\beta)$.
		\item\label{i:equivalence-lag-cob} Any two global Kuranishi charts given by our construction are oriented rel--$C^\infty$ equivalent.
		\item\label{i:cobordism-lag-cob} If $\{J_s'\}_{s\in [-1,1]}\sub \cJ_\tau(X,\omega)$ is another asymptotically constant smooth family of almost complex structures and we define $\wt J'$ as before, then the global Kuranishi charts $\wt \cK$ for $\Mbar_{k,\ell}^{\,J}(\wt\beta)$ and $\wt \cK'$ for $\Mbar_{k,\ell}^{J'}(\wt\beta)$ can be chosen to be oriented cobordant.
	\end{enumerate} 
\end{lemma}

\begin{proof} The proof is very similar to the one of Theorem \ref{thm:gkc-equivariant-existence}, so we will only sketch it and point out the differences. We may assume as there that $k = \ell = 0$. By \cite[Lemma~2.7]{HH25}, we can find an $\wt L$-adapted polarisation $\cO_{\wt X}(1)$. Using a homotopy, we may assume it is constant near $\del\wt X$. Thus, we can define a base space $\wt \cB_d = \Mbar^{\,J_0,d}(\bC P^d,\bR P^d)$ as in \textsection\ref{subsec:construction-review}. Fix a $\wt J$-linear connection $\conn^{\wt X}$ on $\wt X$. As the moduli space is compact, we can choose the integers $p$ and $k$ in the definition of the auxiliary datum \eqref{eq:aux-dat} sufficiently large so that $$\alpha = (\conn^{\wt X},\pr_1^*\cO_X(1),\cU,\lambda,r)$$ 
	defines an unobstructed auxiliary datum for $\Mbar_{k,\ell}(\wt\beta)$, where $\cU$ is the obvious generalisation of a good covering to moduli spaces of maps with boundary on a Lagrangian cobordism. We may, moreover, choose $r$ sufficiently large so that $\alpha$ restricts to an unobstructed auxiliary datum for each $\Mbar_{g,h}^{J_s,\beta}(X,\wt L|_{\{s\}})$ whenever $|1-|s||< 2\epsilon$. In the case of \eqref{i:trivialisable}, we may choose $\cO_{\wt X}(1)$ to be the pulback of an $L$-adapted polarisation that tames each $J_s$. Furthermore, we can find $r$ so that $\alpha$ induces an unobstructed auxiliary datum for any $s \in [-1,1]$, which is well-defined in this case.\par
	Given such an unobstructed auxiliary datum $\alpha$, the arguments of \cite[\textsection 2.2]{HH25} now carry over verbatim as does \cite{HH25}, which shows that $\wt\eva$ can be chosen to be a relatively smooth submersion. Then, Theorem~C.29 op cit allows us to equip the chart with a smooth structure that makes the evaluation map and the map to the base space smooth. Choosing a mollification of the obstruction section as in \cite[Lemma~2.26]{HH25} completes the proof of the first two claims. Since $\bR\sub \bC$ is spin, any choice of relative spin structure $\wt V$ for $\wt L$ restricts to a relative spin structure of $L^\pm_i$. Thus, the claim about orientations follows from \cite[\textsection2.3]{HH25}.
	 
	In the Lagrangian pseudoisotopy case, \eqref{i:trivialisable}, our additional assumption on $\alpha$ implies that the canonical map to $[-1,1]$ is a relatively smooth submersion. By \cite[Theorem~C.29]{HH25}, we can choose the smooth structure so that the map to $[-1,1]$ becomes a smooth submersion.
	The assertions in \eqref{i:equivalence-lag-cob} and \eqref{i:cobordism-lag-cob} are shown as in \cite[\textsection2.5]{HH25}.
\end{proof}

Suppose now $\wt\cK_* =\{\wt \cK^{\wt \beta}_{k,\ell}\}$ is a system of global Kuranishi charts as in Lemma~\ref{thm:cobordism-global-chart-existence}. By the same proof as in \cite[Theorem~4.1]{HH25}, it can be extended to a system of cubical cobordisms $\{\wt\cKc_\Gamma\}_\Gamma$, ie, a system of global Kuranishi charts for the moduli spaces $\Mbar_\Gamma(\wt X,\wt L)$ with the same properties as described in~\S\ref{subsec:cubical-cobordisms}, replacing $L$ by $\wt L$ and with $\wt\cKc_\Gamma$ restricting to a system of cubical cobordisms for $L^-$ respectively $L^+$ over $-1$, respectively $1$. 

\begin{corollary}\label{cor:thom-system-lag-cob} Given any Thom system $\{\eta^{\pm}_{\Gamma}\}$ for the cubical cobordisms $\{\cK^{\scale{\square}{0.6},\pm}_\Gamma\}$, there a system of Thom forms $\{\wt\eta_\Gamma\}$ extending the system $\eta^\pm$ as in Lemma~\ref{lem:thom-form-and-forgetful}.
\end{corollary}

	\subsection{Operations and pseudo-isotopies from Lagrangian cobordisms}
Given a Lagrangian cobordims $\wt L$, we define a generalisation of the $\fq$-operations; see also \cite[\textsection 4]{ST16}. The relevant Novikov ring is given by:
\begin{equation}
	\Lambda_{\wt L} := \left\{  \sum_{i = 0}^{\infty} a_iQ^{\beta_i} | a_i \in \mathbb{R}, \; \beta_i \in H_2(\wt X, \wt L), \; \lim_{i\to \infty } \omega(\beta_i) = \infty \right\}\\  
\end{equation}
For the Novikov ring keeping track of closed counts, we still take the ring $\Lambda_c$, noting that there is a canonical isomorphism $H_2(\wt X) \cong H_2(X)$.
Let $f_\pm\cl \Lambda_{\LL^\pm} \rightarrow \Lambda_{\wt L}$ denote the ring homomorphisms induced by the maps $(X,\LL^\pm) \hookrightarrow (\wt X, \wt L)$.

Assume from now on that $\wt J$ is cylindrical, ie, that $\wt J$ is invariant under translation in the imaginary direction of $\CC$. Fix a Thom system $\{\wt \eta_{k,\ell}^\beta\}$ for $\Mbar(\wt X,\wt L)$ (Corollary \ref{cor:thom-system-lag-cob}).
The same construction as in \textsection\ref{subsec:q-operations} then yields operations \begin{equation*}
	\wt\fq^{\wt X}_{k,\ell}\cl \Omega^*(\wt L;\Lambda_{\wt L})^{\otimes k} \otimes \Omega^*(\wt X;\Lambda_{\wt L})^{\otimes \ell} \rightarrow \Omega^*(\wt L;\Lambda_{\wt L})
\end{equation*}
and 
\begin{equation*}
	\wt\fq^{\wt X}_{-1,\ell}\cl \Omega^*(\wt X; \Lambda_{\wt L})^{\otimes \ell} \rightarrow \Lambda_{\wt L}
\end{equation*}
and closed operations
\begin{equation*}
	\wt{\mathfrak{q}}^{\wt X}_{\emptyset, \ell} \cl \Omega^*(\wt X;\Lambda_c)^{\otimes \ell} \rightarrow \Omega^*(\wt X;\Lambda_c).
\end{equation*}
We change the domain of the operations $\wt \q^{\wt X}$ to get operations $\wt \q$, which are better suited to prove our results. 
Let $\pi\cl \wt X \rightarrow X \times [-1,1]$ be the projection. 
Then, define for $\alpha \in \Omega^*(\wt L)^{\otimes k}$ and $\gamma \in \Omega^*(X \times [-1,1])^{\otimes \ell}$ the maps
	\begin{align}
		\wt \fq_{k,\ell}(\alpha;\gamma) &:= \fq_{k, \ell}^{\wt X}(\alpha; \pi^* \gamma) \qquad
	\end{align}
	for $k \ge -1$ and the operation 
	$$\wt\fq_{\emst,\ell}\cl \Omega^*(X\times [-1,1];\Lambda)^{\otimes\ell}\to \Omega^*(X\times [-1,1];\Lambda)$$ 
	by using the families $\Mbar_{\emst,\ell}^{\{J_s\}_s}(\beta) = [-1,1]\times_{\bC_1}\Mbar_{\emst,\ell}^{\wt J}(i_*\beta)$ of moduli spaces, where $\beta\in H_2(X;\bZ)$.

\begin{remark}
	Note that our $\wt \q_{-1}$ operations differ from those of Solomon-Tukachinsky. In their case they take refined operations which land in $\Omega^*([-1,1], \Lambda)$, which they can do as they only consider Lagrangian pseudoisotopies. By composing with the integration map $\Omega^*([-1,1], \Lambda) \rightarrow \Lambda$ one obtains our operations from theirs.
\end{remark}

Let $j_\pm \cl L_*^\pm \hkra \wt L$ be the inclusions of $L_*^\pm \times \{\pm 1\}$ and let $i_\pm$ denote the inclusions $X \hookrightarrow X \times \bR$ as $X\times\{\pm 1\}$.
By \cite[Lemma~4.2.1]{BC09b}, all holomorphic curves which intersection $\{\pm1 \}\times i\RR$ are necessarily constant in the $\CC$ direction. Therefore we have

\begin{lemma}\label{lem:pullback-intertwines} Let $\q^\pm$ denote the $\q$-operations for $\bL_\pm$. Then, for any 
	integers $k \geq -1$ and $\ell\geq 0$, we have 
	\begin{equation*}
		j_\pm^*\wt \q_{k,\ell}(\alpha,\eta) \;=\; f_\pm\q^\pm_{k,\ell}(j_\pm^*\alpha,i_\pm^*\eta)\qquad \qquad
		i_\pm^*\wt \q_{\emptyset,\ell}(\wt \eta) \;=\; \q_{\emptyset,\ell}^{J_\pm}(i_\pm^*\wt \eta),
	\end{equation*}
	where we indicate the complex structures $J_\pm$ used in the closed operation.\qed
\end{lemma}

\begin{proposition}[Structure equation for $\wt \q_{k,\ell}$]
	\label{boundary of  wt q operation}
	For any $\wt \al$ and $\wt \gamma$ and $k \geq 0$ we have:
	\begin{align*}
		0 = &\sum_{\substack{ P \in S_3[k] \\ (2:3) = \{ j \}}}  (-1)^{|\wt \gamma^{(1:3)}| + 1} \wt\q_{k,\ell} (\wt \alpha;\wt \gamma^{(1:3)} \otimes d \wt \gamma_j \otimes \wt \gamma^{(3:3)})\\
		+&\;\;\sum_{\substack{ P \in S_3[k] \\ I \sqcup J = [\ell]}} (-1)^{i(\wt \alpha, \wt \gamma, P, I)} \wt \q_{k_1+1+k_3,|I|}(\wt \alpha^{(1:3)}\otimes \wt \q_{k_2,|J|}(\wt \alpha^{(2:3)};\wt \gamma^{J})\otimes \wt \alpha^{(3:3)};\wt \gamma^I).
	\end{align*}
\end{proposition}

\begin{proof}
	Since the vertical boundary of the evaluation map $\wt\eva_0^\Gamma \cl \wt\cTc_\Gamma \to \wt L$ is given by the analogous terms as the vertical boundary of $\eva_0^\Gamma \cl \cTc_\Gamma \to L$, the proof of Proposition~\ref{prop:q structure} carries over.
\end{proof}

\begin{cor}
	The operations $\wt\fq_{k,0}$ equip $\Omega^*(\wt L;\Lambda_{\wt L})$ with an $A_\infty$-structure.
\end{cor}

By the same proof as in \cite{ST16}, we have
\begin{proposition}[{\cite[Proposition~4.19]{ST16}}]
	For $\wt \eta \in \Omega^*(X\times [-1,1])^{\otimes \ell}$ we have $
		\wt \q_{\emptyset,\ell}(d \eta) = d \wt \q_{\emptyset,\ell}$.
\end{proposition}

Combining this with Stokes' theorem yields the following.
\begin{lemma}
	Let $\pi_X\cl X \times [-1,1] \rightarrow X$ denote the projection. Then,
	\begin{equation}
		d (\pi_X)_*\wt \q_{\emptyset,\ell}(\wt \eta) - (\pi_X)_*\wt \q_{\emptyset,\ell}(d \wt\eta)= (-1)^{|\wt \eta|+1} \left( \q^{J_+}_{\emptyset,\ell}(i_+^*\wt \eta) - \q^{J_-}_{\emptyset,\ell}(i_-^*\wt \eta) \right).
	\end{equation}
	for any $\wt \eta \in \Omega^*(X \times [-1,1])$.
\end{lemma}

\smallskip
The structure equation for the operations $\wt q_{-1,\ell}$ is different from those for $\q_{-1,\ell}$, as the integration map $\int\cl \wt \cT_{0,\ell} \rightarrow \RR$ has an additional vertical boundary stratum given by $\cT^{\pm}_{0,\ell}$.
\begin{proposition}
	\label{prop: structure equation bulk deformed q -1 tilde} We have
	\begin{multline*}
		\wt \q_{-1,\ell}(d \wt \gamma) =  (-1)^{n+|\wt\gamma| + 1}(f_+\q_{-1,\ell}^{J_+}(i_+^*\wt\gamma) - f_-\q_{-1,\ell}^{J_-}(i_-^*\wt\gamma))\\
		+\frac{1}{2}\sum_{I\sqcup J=\{1,\ldots,\ell\}}
		(-1)^{\iota(\wt \gamma;I)}
		\langle \wt \q_{0,|I|}(\wt \gamma_I),\wt \q_{0,|J|}(\wt \gamma_J)\rangle_{\wt L}
		+ (-1)^{|\wt\gamma|+1}\int_{\wt L} \pi_I^* \wt \q_{\emptyset,\ell}(\wt \gamma).
	\end{multline*}
\end{proposition}

\begin{proof}
	The proof is similar to that of \cite[Proposition~2.5]{ST16}, except for the additional terms $f_+\q_{-1,\ell}^{J_+}(i_+^*\wt\gamma) - f_-\q_{-1,\ell}^{J_-}(i_-^*\wt\gamma)$, which come from the additional vertical boundary. The sign in front of this arises as $|\wt \q_{-1,\ell}(\wt \gamma)| + 1 \equiv n+ |\wt \gamma| +1$. Alternatively, one can compute the sign by taking \cite[Proposition~4.4]{ST16} and integrating over $[-1,1]$.
\end{proof}

Given a bulk deformation parameter $\wt \gamma \in \Omega^*(X\times [-1,1]; \mathcal{I}Q)$ and $\wt b \in \Omega^*(\wt L, \mathcal{I}R)$ one can define bulk-deformed $\wt \q$ operations, and bounding pairs exactly as in \textsection\ref{section: bulk-deformations and bounding cochains}. Lemma \ref{lem:pullback-intertwines} shows that the pullback of a bounding pair $(\wt \ga, \wt b)$ to $(X, L_*^\pm)$, denoted by $(\gamma_\pm, b_\pm)$ is again a bounding pair. Moreover, the Lagrangian superpotential satisfies $$f_\pm c^{\pm} = j^*_{\pm}\wt c.$$

\begin{remark}
	Technically $(\gamma_\pm,b_\pm)$ are only bounding pairs for the operations $f_\pm\q^{J_\pm}$ on $\Omega^*(\LL^\pm;\Lambda_{\LL_\pm})$. We will only need these operations, so will be content with this abuse of notation.
\end{remark} 

The bulk-deformed structure equations are then given as follows. 

\begin{proposition}\label{prop: bulk deformed q tilde operation} Given $\wt \eta \in \Omega^*(X \times [-1,1])$ we have
	\begin{equation*}\label{deformed q tilde empty}
		 d(\pi_X)_* \wt \q_{\emptyset,\ell}^{\,\wt \gamma}(\wt \eta) - (\pi_X)_*\wt \q_{\emptyset,\ell}^{\,\wt \gamma}(d \wt \eta) = (-1)^{|\wt \eta|+1} \left(\q_{\emptyset,\ell}^{\gamma_0}(i_1^* \wt \eta) - \q_{\emptyset,\ell}^{\gamma_1}(i_0^* \wt \eta)\right)
	\end{equation*}
	and
	\begin{multline*}\label{deformed q tilde -1} \wt \q_{-1,\ell}^{\,\wt \gamma, \wt b}(d\wt \eta) =  (-1)^{n+|\wt \eta|+1}\lbr{f_+\q_{-1,\ell}^{\gamma_+, b_+}(i_+^*\wt \eta) - f_-\q_{-1,\ell}^{\gamma_-, b_-}(i_-^*\wt \eta)} \\ +\frac{1}{2}\sum_{I\sqcup J=\{1,\ldots,\ell\}}
		(-1)^{\iota(\wt \eta;I)} \langle\, \wt \q_{0,|I|}^{\,\wt \gamma, \wt b}(\wt \eta_I),\,\wt \q_{0,|J|}^{\,\wt \gamma, \wt b}(\wt \eta_J)\,\rangle_{\wt L}
		+ (-1)^{|\wt\eta|+1}\int_{\wt L} \pi_I^* \wt \q_{\emptyset,\ell}^{\,\wt \gamma}(\wt \eta).
	\end{multline*}
\end{proposition}

 A special class of Lagrangian cobordism yields the invariance of the open Gromov--Witten invariants, Theorem~\ref{thm:invariance-ogw}, which is the main theorem of this subsection.
\begin{definition}\label{def:pseudoisotopy}
	A \emph{(cyclic unital) pseudo-isotopy} of (cyclic unital) curved $A_\infty$ algebra structures on $\Omega^*(L)$ is a (cyclic unital) curved $A_\infty$ algebra structure on $\Omega^*(L \times [-1,1])$.
\end{definition}

Fukaya showed that the notion of being pseudo-isotopic is stronger than being quasi-equivalent.

\begin{theorem}[{\cite[Theorem~8.2]{Fu10}}]
	\label{thm:pseudoisotopy induces QI}
	A (cyclic unital) pseudo-isotopy between two $A_\infty$ algebras $\scA_-$ and $\scA_+$ induces a (cyclic unital) quasi-equivalence $\scA_- \simeq \scA_+$.
\end{theorem}

Our main result on independence from choices is a straightforward consequence.

\begin{proposition}\label{thm:trivialisable cobordism induces quasi-equivalence}
	The unital cyclic curved $A_\infty$-algebra $CF^*(L)$ is independent up to cyclic unital pseudo-isotopy of the choice of 
	\begin{enumerate}[\normalfont \arabic*),leftmargin=20pt, ref=\arabic*]
		\item\label{choice-thom-system} Thom system,
		\item\label{choice-auxiliary-data} unobstructed auxiliary data,
		\item\label{choice-isotopy} $L$ in its Hamiltonian isotopy class,
		\item\label{choice-acs} tame almost complex structure.
	\end{enumerate}
\end{proposition}

\begin{proof}
 Given a system $\{\cK_{\mathsf{a}} = \cK_{\alpha_{\mathsf{a}}}\}_{\mathsf{a}}$ of global Kuranishi charts as in Proposition~\ref{prop:cubical-cobordisms} and two choices of Thom systems $\{\eta^\pm_{\mathsf{a}}\}$, the auxiliary datum $\alpha_{\mathsf{a}}$ is also an unobstructed auxiliary datum for $\Mbar_{\mathsf{a}}(\wt\beta)$, where $\wt L = L \times[-1,1]$ is the trivial Lagrangian cobordism and $\wt\beta$ is the image of $\beta$ in $H_2(\wt X,\wt L;\bZ)$. By Corollary \ref{cor:thom-system-lag-cob}, we may equip the associated cubical cobordism with a Thom system $\wt \eta$ which restricts to $\eta^\pm$ on the boundary. The operations $\wt \q$ yield the required pseudo-isotopy, proving \eqref{choice-thom-system}.\par
	Given two systems $\{\cK_{{\mathsf{a}}}= \cK_{\alpha_{\mathsf{a}}}\}$ and $\{\cK'_{\mathsf{a}}= \cK_{\alpha'_{\mathsf{a}}}\}$, we recall the construction of \cite[\S2.5]{HH25} of a global Kuranishi chart
	\begin{equation*} \wt\cK_{\mathsf{a}} = (G_{\mathsf{a}}\times G'_{\mathsf{a}},\wt\cT_{\mathsf{a}}/\wt\cB_{\mathsf{a}},\wt\cE_{\mathsf{a}},\wt\obs_{\mathsf{a}})\end{equation*}
	for $[0,1]\times\Mbar_{\mathsf{a}}$ that is a cobordism from $\cK_{\alpha_{\mathsf{a},0}}$ and $\cK_{\alpha_{\mathsf{a},0}}$ up to stabilisation and group enlargement. Let 
	$$\wt\cB_{\mathsf{a}}\sub \Mbar_{\lc{\mathsf{a}}}^{J_0\times J_0\, (d_0,d_1)}(\bC P^{d_0}\times \bC P^{d_1},\bR P^{d_0}\times \bR P^{d_1})$$
	be the locus of regular maps whose complex doubles have no non-trivial automorphisms. Write $\wt\cC_{\mathsf{a}}\to \wt\cB_{\mathsf{a}}$ for its universal family. Given $\wt\varphi \in \wt\cB_{\mathsf{a}}$, we write $\varphi_i$ for the composition of $\wt\varphi$ with the projection to the respective factor. Defining $\wt\cZ_{\mathsf{a}}$ just as we defined $\cZ_{\mathsf{a}}$ in \eqref{eq:family}, we let
	\begin{equation*}\wt\cT\sub \set{(t,\wt\varphi,u,\eta_0,\eta_1)\;\big|\; {(\wt\varphi,u)\in \wt\cZ_{\mathsf{a}}, \delbar_J u + (1-t)\lspan{\eta_0}\g d\varphi_0 + t\lspan{\eta_1}\g d\varphi_1 = 0}}\end{equation*}
	be the subset of curves satisfying that the linearisation of the perturbed Cauchy--Riemann equation is regular. The obstruction bundle $\wt\cE\to \wt \cT$ is 
	\begin{equation*}\wt\cE_{(\wt\varphi,u,\eta,\eta')} = H^0(C,\varphi_0^*E_0)\oplus \fp_{\mathsf{a}_0}\oplus H^0(C,\varphi_0^*E_1)\oplus \fp_{\mathsf{a}_1}\end{equation*}
	and the obstruction section is (a mollification of) the analogue of \eqref{eq:obstruction-section}. It is a straightforward verification to see that $\cK_{\mathsf{a}_i}$ is related to the fibre of $\wt\cK_{\mathsf{a}}$ over $i$ by group enlargement and stabilisation. It follows from the explicit construction of the cubical cobordisms in \cite[Theorem~4.1]{HH25} that the same remains true for the associated cubical cobordisms.
	Picking any Thom system $\eta^i_*$ for the cubical cobordism associated to the choices of auxiliary data, we may extend them to a Thom system for the cubical cobordism $\{\wt\cKc_\Gamma\}$ by Corollary~\ref{cor:thom-system-lag-cob}. The operations $\wt \q$ then yield the required pseudo-isotopy.
	
	By Example~\ref{ex:suspension-hamiltonian-isotopy} respectively the trivial Lagrangian cobordism where only the almost complex structure varies, we can describe Hamiltonian isotopies and isotopies of tame almost complex structures in terms Lagrangian cobordisms. As the space of tame almost complex structures is contractible, the result follows.
\end{proof}

\begin{remark}
	The pseudo-isotopy arising from the choice of Thom system for $\wt L$ is non-unique, but any two choices of such lead to pseudo-isotopic pseudo-isotopies and so forth.
\end{remark}

We adapt {\cite[Definition~2.23]{ST21} to our context. As before, $L$ is a relatively spin Lagrangian.
	
\begin{definition}
	 Two bounding pairs $(\gamma_\pm,b_\pm)$ for (possibly different) choices of auxiliary data, Thom systems, and tame almost complex structures $J_\pm$ are \emph{gauge equivalent} if there exist auxiliary data, Thom systems and an isotopy of tame almost complex structures for the Lagrangian $L \times [-1,1]$, and $\wt \ga \in \Omega^*(X \times [-1,1]; \mathcal{I}Q)$, $\wt b \in \Omega^*(L \times [-1,1]; \mathcal{I}R)$ with 
	\begin{equation*}
		i_\pm^*\wt \ga \; =\; \ga_\pm, \;\; j_\pm^*\wt b\; =\; b_\pm, \;\; d\wt \gamma = 0, \;\; |\wt \ga| = 2,
	\end{equation*}
	such that the associated $\wt \q$ operation satisfies 
	$\wt \q^{\wt b, \wt \ga}_{0,0}\; =\; \wt c \cdot 1$
	for some $\wt c \in \mathcal{I}R$ with $|\wt c| = 2$.
\end{definition}

The proof of Theorem~3.1 \cite{Fuk11} shows that pseudo-isotopies induce an isomorphism between the sets of bounding cochains up to gauge-equivalence.

\begin{lemma}
	\label{lem:pseudo-isotopy induces gauge equivalence}
	Let $A_\pm$ be pseudo-isotopic curved $A_\infty$ algebras, then a bounding cochain $b_-$ induces a gauge-equivalent bounding cochain $b_+$.
\end{lemma}

\begin{proof}
	Let $A_t$ be the $A_\infty$ algebra at $t$. Let $\phi_t\cl A_- \rightarrow A_t$ be the quasi-equivalence induced by the pseudo-isotopy (constructed in \cite[Theorem~8.2]{Fu10}), and $(\phi_t)_*$ the induced map on spaces of bounding cochains (constructed in \cite[Section~4.3.2]{FOOO09}). Then, set $b(t) = (\phi_t)_*b_-$. This provides the required gauge equivalence.
\end{proof}

\begin{corollary}
	Let $\gamma, \gamma'$ be any bulk-deformation parameters with $[\gamma] = [\gamma'] \in H^*(X;Q)$. Then, if there exists a bounding cochain $b$ for $\gamma$ for some choice of auxiliary data, Thom system and tame almost complex structure, then, for any other choice of auxiliary data, Thom system and tame almost complex structure, there exists a bounding cochain $b'$ for $\gamma'$ so that $(b,\gamma)$ is gauge equivalent to $(b',\gamma')$.
\end{corollary}

\begin{proof}
	Proposition \ref{thm:trivialisable cobordism induces quasi-equivalence} combined with \cite[Theorem~2]{ST16} shows that the $A_\infty$ algebras $\m^\gamma$ and $\m^{\gamma'}$ are related by a cyclic unital pseudo-isotopy. The result thus follows from Lemma \ref{lem:pseudo-isotopy induces gauge equivalence}.
\end{proof}

\begin{theorem}\label{thm:invariance-potentials}
	The tensor potential and the relative potential are invariant under gauge equivalence.
\end{theorem}

\begin{proof}  The proofs of Theorem 1 and Theorem 4 in \cite{ST23} carry over verbatim.
\end{proof}

This allows us to conclude with the following invariance statement. 

\begin{theorem}\label{thm:invariance-ogw}
	The open Gromov-Witten invariants defined in this paper are independent of 
	\begin{itemize}[leftmargin=20pt]
		\item  the choice of Thom system,
		\item  the choice of unobstructed auxiliary data used to construct the global Kuranishi charts,
		\item  the choice of almost complex structure,
		\item  the representative of $L$ in its Hamiltonian isotopy class,
		\item  the representative of the bounding pair in its gauge equivalence class.
	\end{itemize} 
\end{theorem}

\subsection{Invariance}
Now, consider two generalised Lagrangian branes $\LL^\pm$ with bounding pairs $(\gamma^\pm, b^\pm)$ so that $\gamma^+ = \gamma^-$. Suppose that $\wt L$ is a Lagrangian cobordism with ends $\LL^\pm$ so that the bulk deformation $\pi_X^* \gamma$ is unobstructed by the bounding cochain $\wt b$ (with $i_\pm^* \wt b = b^\pm$), for a constant complex structure $J$. Consider the map 
$$\Phi\cl \normalfont\text{Cone}(\LL^-) \rightarrow \normalfont\text{Cone}(\LL^+)$$ 
given by 
$$\Phi(\alpha, a) \; =\; (\alpha, a + \int_{\wt L} \pi_X^* \alpha),$$ 
where $\pi_X$ is the projection $\wt X \rightarrow X$. An easy computation shows that $\Phi$ is an isomorphism of chain complexes.
This subsection is devoted to proving the comparison of open GW invariants of $\bL^+$ and $\bL^-$ as stated below. 

\begin{theorem}
	\label{thm: OGW are respected by cobordism}
	Let $\normalfont\text{OGW}^{\LL^+}$ be defined using a given map $P_{\RR}^+\cl H^*(C(\fii^+_\RR)) \rightarrow \coker(\fii^+_\RR)$. Then, for each $\beta \in H_2(\wt X, \wt L)$ and $k \geq 0$, we have
	\begin{equation*}
		\sum_{f_+(\beta_+)  =\beta} \normalfont\text{OGW}^{\LL^+}_{\beta_+,k,\ell} \; =\; \sum_{f_(\beta_-) =\beta}{\normalfont\text{OGW}}^{\LL^-}_{\beta_-,k,\ell}\cl H^*(X)^{\otimes \ell} \rightarrow \RR,
	\end{equation*}
	where the invariants $\normalfont\text{OGW}^{\LL^-}$ are defined using the map $P_{\RR}^- := P_\RR^+ \circ \Phi$.
\end{theorem}

\begin{remark}
	In particular, as $\omega(\beta) = \wt\omega(f_\pm \beta)$, we find that for each particular energy level, the open Gromov-Witten invariants agree.
\end{remark}

\begin{remark}
	Suppose that $f_+^{-1}(\{\beta\}) = \emptyset$. Then, $\sum_{f_-\beta_-  =\beta}{\text{OGW}}^{\LL^-}_{\beta_-,k} = 0$.
\end{remark}

Theorem \ref{thm: OGW are respected by cobordism} is a direct consequence of the following generalisation of \cite[Theorem~1]{ST23}.

\begin{theorem}
	\label{thm: relative potential is preserved}
	If the assumptions of Theorem \ref{thm: OGW are respected by cobordism} are satisfied, then 
	$$f_+ \Psi^{\LL^+}(\gamma,b_+) \; =\; \Phi \left( f_- \Psi^{\LL^-}(\gamma, b_-) \right).$$
\end{theorem}

\begin{proof} Under these assumptions, 
\begin{equation*}\label{} \pi_X^* \q_{\emptyset,0}^\ga \; =\; \pi_I^* \,\wt \q_{\emptyset,0}^{\,\gamma}, \end{equation*} 
where $\pi_I\cl \wt X \rightarrow X \times [-1,1]$ is the projection. Thus,
	\begin{align*}
		\Phi\left( f_- \Psi^{\LL^-}(\gamma, b_-) \right) &= \lbr{\q_{\emptyset,0}^\gamma, (-1)^{n+1}f_-\q^{\ga,b^-}_{-1,0} + \int_{\wt L} \pi_X^* \q_{\emptyset,0}^\ga}\\
		&= \lbr{\q_{\emptyset,0}^\gamma, (-1)^{n+1}f_-\q^{\ga,b^-}_{-1,0} + \int_{\wt L} \pi_I^* \wt \q_{\emptyset,0}^\ga}
	\end{align*}
	By Proposition \ref{prop: bulk deformed q tilde operation}, the integral is given by  \begin{align*}
		\int_{\wt L} \pi_I^* \wt \q_{\emptyset,0}^\ga &= (-1)^{n+1}\lbr{f_+\q_{-1,0}^{\gamma, b_+} - f_-\q_{-1,0}^{\gamma, b_-} + \frac{1}{2} \langle \,\wt \q_{0,0}^{\,\wt \gamma, \wt b},\,\wt \q_{0,0}^{\,\wt \gamma, \wt b}\,\rangle_{\wt L}}\\
		&= (-1)^{n+1}\lbr{f_+\q_{-1,0}^{\gamma, b_+} - f_-\q_{-1,0}^{\gamma, b_-} + \frac{1}{2} \langle \wt c \cdot 1,\wt c \cdot 1\rangle_{\wt L}}\\
		&=(-1)^{n+1}(f_+\q_{-1,0}^{\gamma, b_+} - f_-\q_{-1,0}^{\gamma, b_-}).
	\end{align*}
	Therefore,
	\begin{equation*}
		f_+ \Psi^{\LL^+}(\gamma,b_+) = (\q_{\emptyset,0}^\gamma, (-1)^{n+1} f_+\q^{\ga,b^+}_{-1,0}) = \Phi\left( f_- \Psi^{\LL^-}(\gamma, b_-) \right).
	\end{equation*}
\end{proof}
A similar computation shows the following generalisation of \cite[Theorem~4]{ST23}.

\begin{theorem}
	The map $\Phi$ intertwines the tensor potential, that is, 
	\begin{equation*}
		\begin{tikzcd}
			H^*(C(\fii^-)) \arrow[r, "\nun"] \arrow[d,"\Phi"] &H^*(C(\fii^-)) \arrow[d,"\Phi"]\\
			H^*(C(\fii^+)) \arrow[r, "\nun"] &H^*(C(\fii^+))
		\end{tikzcd}
	\end{equation*}
	commutes.
\end{theorem}

\begin{proof}	
	By Lemma \ref{lem:pullback-intertwines}, we have $c^+ = \wt c = c^- =: c$, so that
	\begin{align}
		\nun^+(\Phi(\eta,a)) - \Phi(\nun^-(\eta,a)) \; =\; (0,(-1)^{n+1}(\q_{-1,1}^{\ga,b^+} - \q_{-1,1}^{\ga,b^-}) - c \int_{\wt L} \pi_X^* \eta - \int_{\wt L} \pi_X^* \q_{\emptyset,1}^\ga(\eta) ).
	\end{align}
	Now observe that, by Proposition \ref{prop: structure equation bulk deformed q -1 tilde}, we have
	\begin{align*}
		(-1)^{|\eta|}(\wt{\q}_{-1,1}^{\ga,\wt b}(d\pi_X^*\eta) &= (-1)^{n+1}(\q_{-1,1}^{\ga,b^+} - \q_{-1,1}^{\ga,b^-}) + \frac{1}{2} \langle \,\wt \q_{0,1}^{\ga,\wt b}(\pi_X^*\eta),\,\wt \q_{0,0}^{\ga,\wt b} \,\rangle_{\wt L} \\&\qquad+\frac{(-1)^{|\eta|}}{2} \langle\, \wt \q_{0,0}^{\ga,\wt b},\,\wt \q_{0,0}^{\ga,\wt b}(\pi_X^*\eta)\, \rangle_{\wt L} - \int_{\wt L} \pi_X^* \q_{\emptyset,1}(\pi_X^* \eta)
	\end{align*}
	Using Proposition \ref{cyclic symmetry}, the bounding cochain equation, Proposition \ref{degree property}, and Proposition~\ref{energy zero property}, we deduce
	\begin{equation*}
		\frac{1}{2} \langle\, \wt \q_{0,1}^{\ga,\wt b}(\pi_X^*\eta),\,\wt \q_{0,0}^{\ga,\wt b} \,\rangle_{\wt L} +\frac{(-1)^{|\eta|}}{2} \langle\, \wt \q_{0,0}^{\ga,\wt b},\,\wt \q_{0,0}^{\ga,\wt b}(\pi_X^*\eta)\, \rangle_{\wt L}\; =\; -c\int_{\wt L}\pi_X^*\eta
	\end{equation*}
	Thus, let $H\cl C(\fii^-) \rightarrow C(\fii^+)$ be given by 
	$$H(\eta, a) = (0, (-1)^{|\eta| +1}\wt \q_{-1,1}^{\ga,\wt b}(\eta)).$$ 
	The above then shows that 
		$\nun^+(\Phi(\eta,a)) - \Phi(\nun^-(\eta,a)) = [d,H]$.
\end{proof}


The following result requires the additional observation that $\Phi$ respects the Euler grading.

\begin{theorem}\label{thm: cobordism induces isomorphism of TE-structures}
	The map $\Phi$ induces an isomorphism of TE-structures\begin{equation}
		(QH^*(X,\LL^-;Q\otimes_{\Lambda_{\LL^-}} \Lambda_{\wt L})[[u]], \nabla^{\LL^-}) \cong (QH^*(X,\LL^+;Q\otimes{\Lambda_{\LL^+}} \Lambda_{\wt L})[[u]], \nabla^{\LL^+}).
	\end{equation}
	When $\LL^-$ (and hence $\LL^+$) is null-homologous, $\Phi$ makes the following diagram commute:
	\begin{equation}
		\begin{tikzcd}
			0 \ar[r] &Q \otimes_{\Lambda_{\LL^-}} \Lambda_{\wt L}[[u]] \ar[d, equal] \ar[r] & QH^*_-(X, \LL^-; Q \otimes_{\Lambda_{\LL^-}} \Lambda_{\wt L} ) \ar[d, "\Phi"] \ar[r] & QH^*_-(X;Q \otimes_{\Lambda_c} \Lambda_{\wt L})  \ar[d,equal] \ar[r] &0\\
			0 \ar[r] &Q \otimes_{\Lambda_{\LL^-}} \Lambda_{\wt L}[[u]] \ar[r] & QH^*_-(X, \LL^+; Q \otimes_{\Lambda_{\LL^+}} \Lambda_{\wt L}) \ar[r] & QH^*_-(X;Q \otimes_{\Lambda_c} \Lambda_{\wt L})  \ar[r] &0
		\end{tikzcd}
	\end{equation}
\end{theorem}

	\bibliographystyle{amsalpha}
	\bibliography{bib-zero}
	
	\Addresses
	
\end{document}